\documentclass[11pt,a4paper]{article}
\usepackage{geometry}
\usepackage{graphicx}
\usepackage{bm,array}
\usepackage{comment}
\usepackage{caption}
\usepackage{subcaption}
 \usepackage{makecell}
 \usepackage{multirow}
 \usepackage{float}

\usepackage[normalem]{ulem}

\usepackage[pdfpagemode={UseOutlines},bookmarks=true,bookmarksopen=true,
bookmarksopenlevel=0,bookmarksnumbered=true,hypertexnames=false,
colorlinks,linkcolor={blue},citecolor={blue},urlcolor={red},
pdfstartview={FitV},unicode,breaklinks=true]{hyperref}
\hypersetup{urlcolor=blue, colorlinks=true}

\usepackage{setspace}
\usepackage{vmargin}
\setmarginsrb{ 1.0in}  
{ 0.6in}  
{ 1.0in}  
{ 0.8in}  
{  20pt}  
{0.25in}  
{9pt}  
{ 0.3in}  

\usepackage[round, sort,comma,authoryear]{natbib}

\usepackage{authblk}

\usepackage{amsmath}
\usepackage{amsfonts}
\usepackage{amssymb}

\usepackage{multirow}
\usepackage{hhline}
\usepackage{footnote}
\usepackage[ruled,vlined]{algorithm2e}
\usepackage{algorithmic}
\newcommand{\comments}[1]{{\color{blue}\textit{$\#$ #1}}}

\newcommand{\transpose}{{\mbox{\tiny T}}}

\newcommand{\cV}{{\mathcal{V}}}
\newcommand{\cD}{{\mathcal{D}}}
\newcommand{\cK}{{\mathcal{K}}}

\newcommand{\cL}{{\mathcal{L}}}

\newcommand{\cO}{{\mathcal{O}}}

\newcommand{\bbe}{{\textbf{e}}}

\newcommand{\bB}{\textbf{B}}

\newcommand{\bV}{\textbf{V}}
\newcommand{\bY}{\textbf{Y}}
\newcommand{\bx}{\textbf{x}}

\newcommand{\bc}{\textbf{c}}

\newcommand{\bA}{\textbf{A}}

\newcommand{\ba}{\textbf{a}}

\newcommand{\bv}{\textbf{v}}

\newcommand{\bld}{{\pmb{\lambda}}}

\newcommand{\btau}{{\pmb{\tau}}}
\newcommand{\btheta}{{\pmb{\theta}}}
\newcommand{\dg}{\text{diag}}

\newcommand{\bbR}{\mathbb{R}}

\newtheorem{theorem}{Theorem}

\newtheorem{corollary}{Corollary}

\newtheorem{lemma}{Lemma}

\newtheorem{proposition}{Proposition}
\newtheorem{remark}{Remark}

\newenvironment{proof}[1][Proof]{\noindent\textbf{#1.} }{\ \rule{0.5em}{0.5em}}

\newcommand{\WC}{\textsc{\tiny WC}}
\newcommand{\SA}{\textsc{\tiny SA}}

\usepackage{mathtools}

\newif\ifnotes\notestrue
\def\mtien#1{{\color{black}{#1}}}

\def\htien#1{}

\newcommand{\RO}{\textsc{\tiny RO}}
\newcommand{\DET}{\textsc{\tiny DET1}}
\newcommand{\MDET}{\textsc{\tiny DET2}}

\begin{document}
	
\newcolumntype{C}{>{\centering\arraybackslash}p{4em}}

\title{\textbf{Robust Maximum Capture Facility Location under Random Utility Maximization Models}}
\author[1]{Tien Thanh Dam}
\author[2]{Thuy Anh Ta}
\author[3]{Tien Mai}
\affil[1]{\it\small 
ORLab, Faculty of Computer Science, Phenikaa University, Yen Nghia, Ha Dong, Hanoi, VietNam, anh.tathuy@phenikaa-uni.edu.vn}
\affil[2]{\it\small 
ORLab, Faculty of Computer Science, Phenikaa University, Yen Nghia, Ha Dong, Hanoi, VietNam,
thanh.damtien@phenikaa-uni.edu.vn}
\affil[3]{\it\small
School of Computing and Information Systems, Singapore Management University, 80 Stamford Rd, Singapore 178902, atmai@smu.edu.sg
	}

\maketitle

\pagebreak
\begin{center}
	\linespread{1.5}
	\LARGE
	\bf
	{Robust Maximum Capture Facility Location under Random Utility Maximization Models}
\end{center}

\begin{abstract}
We study a robust version of the maximum capture facility location problem in a competitive market, assuming that each customer chooses among all available facilities according to a random utility maximization (RUM) model.  We employ the generalized extreme value (GEV) family of models and assume that the  parameters of the RUM model are not given exactly but lie in convex uncertainty sets. The problem is to locate new facilities to maximize  the worst-case captured user demand. We show that, interestingly,  our robust model preserves the monotonicity and submodularity from its deterministic counterpart, implying that  a simple greedy heuristic can guarantee a $(1-1/e)$ approximation solution. We further show the concavity of the objective function under the classical multinomial logit (MNL) model, suggesting that an outer-approximation algorithm can be used to solve the robust model under MNL to optimality. We conduct experiments comparing our robust method to other deterministic and sampling approaches, using instances from different discrete choice models. Our results clearly demonstrate the advantages of our {robust} model in protecting the decision-maker from worst-case scenarios.         
\end{abstract}

{\bf Keywords:}  
Facilities planning and design, maximum capture, random utility maximization, robust optimization, local search, outer-approximation 


%


\section{Introduction}
Facility location is an active research area in operations research and has been attracting researchers for decades. Facility location problems  play  important roles in many decision-making tasks such as installation  of new retail or service facilities in a market, launching new products to the market, or developing optimal customer segmentation policies. In facility location, a firm  aims at selecting a set of locations  to locate their facilities to maximize a profit or minimize a cost function.  
In this context, to make good decisions, one may need to build a good model to predict customers' behavior with respect to each possible facility location decision. The random utility maximization (RUM) discrete choice framework \citep{McFa78} has become useful in the context due to its flexibility in capturing human behavior. 
To the best of our knowledge, existing works on facility location under RUM all assume that the parameters of the RUM model are known with certainty and ignore any uncertainty associated with  the estimates, with a tacit understanding that the parameters have to be estimated in practice.
 Such an estimation can cause errors and the decision-maker needs to cope with the fact that the estimates of the choice parameters can significantly deviate from their true values. Ignoring such estimation errors would lead to bad decisions, as shown in several studies in the robust optimization literature \citep[see][for instance]{Bertsimas2011theoryRO}. In this paper, we address this uncertainty issue by studying a robust version of the facility location problem under RUM.   

We consider the problem of how to locate new facilities in a competitive market such that the captured demand of users is maximized, assuming that each individual chooses among all available facilities according to a RUM model. This problem is  called as the maximum capture problem (MCP) \citep{BenaHans02}. We formulate and solve the MCP under uncertainty in a robust manner. That is, we assume that customers' behavior is driven  by the well-known generalized extreme value  (GEV) family of models, but the parameters of the RUM model cannot be determined with certainty and belong to some uncertainty sets. These uncertainty sets can represent a partial knowledge of the decision-makers about the RUM model and can be inferred from data. The goal here is to maximize the worst-case expected captured customer demand when the RUM parameters vary in the uncertainty sets. 
We will study theoretical properties and develop algorithms for the robust MCP under any GEV model  and, in particular, the robust MCP under the popular multinomial logit (MNL)  \citep{Trai03}.  

Before presenting our contributions in detail, we note that, when mentioning the ``GEV'' model, we refer to any RUM (or discrete choice) model in the GEV family. This family covers most  of the discrete choice  (or RUM)  models in the literature \citep{Trai03}. 

\noindent

\textbf{Our contributions:}
We study a robust version of the MCP under a GEV model, assuming that the parameters of the GEV model are not known with certainty but can take any values in some uncertainty sets, the uncertainty sets are customer-wise independent, and 
the objective is to maximize the worst-case expected captured customer demand. We will show that, under our uncertainty settings, the inner minimization problem can be solved by convex optimization. We then leverage the properties of the GEV family to show that the worst-case objective function is monotonic and submodular for any GEV model, noting that the monotonicity and submodularity 
have been shown for the deterministic MCP \citep{Dam2021submodularity} and in this work, we show that the robust model preserves both properties. Here, it is important to note that a robust submodular
maximization problem is generally inapproximable, i.e. there is no
polynomial-time algorithm that can guarantee a positive
fraction of the optimal value,  unless {\sf P = NP} \citep{krause2008robust}.
Our results, however, show that, in the context of the MCP, a simple  polynomial-time greedy algorithm can achieve $(1-1/e)$ approximation solutions.  

The monotonicity and submodularity of the robust problem imply that  the robust MCP, under a cardinality constraint,  always admits a $(1-1/e)$ approximation algorithm \citep{Nemhauser1978analysis}. That is, we can simply start from an empty set and iteratively select locations, one at a time, taking at each step the location that increases the worst-case objective function the most, until the maximum capacity is reached. A solution from this simple procedure will yield an objective value being at least $(1-1/e)$ times the optimal value of the robust problem. We then further adapt the local search procedure proposed by \cite{Dam2021submodularity} to efficiently solve the robust MCP under GEV models.  Our results generally hold for any RUM model in the GEV family,  and under any convex uncertainty sets.  We further consider the robust MCP under the MNL model and show that, under the assumption that the uncertainty sets are independent over customer zones, the robustness  preserves the concavity of the relaxation of the worst-case objective function, implying that an outer-approximation algorithm can be used to exactly solve the MCP under MNL. A multicut outer-approximation algorithm  is then presented for this purpose. 
We finally conduct experiments based on the MNL and nested logit (two  most popular RUM models) to demonstrate  the advantages of our robust model in protecting decision-makers from worst-case scenarios, as compared to other deterministic and sampling-based baselines.

\textbf{Literature review:}
The GEV family \citep{McFa78,Trai03} covers many popular RUM models in
the demand modeling and operations research literature. In this family, the simplest and most popular member is the MNL \citep{McFa78}
and it is well-known that the MNL model retains the independence from irrelevant alternatives (IIA) property, which is often regarded as a limitation of the MNL model. The literature has seen several other GEV models that  relax this property and provide 
flexibility in modeling the correlation between choice alternatives. Some examples are the nested logit \citep{BenA73,BenALerm85}, the cross-nested logit \citep{VovsBekh98}, and network GEV \citep{Daly2006,MaiFreFosBas15_DynMEV} models. The cross-nested and network GEV models are considered as being fully flexible as they can approximate any RUM model \citep{Fosgerau2013}. It is worth noting that in the context of descriptive representation (i.e., modeling human behavior) the MNL and nested logit models are found the most empirically applicable among existing RUM models, but in prescriptive optimization (i.e., decision-making), apart from the MNL, the use of other GEV models is limited due to their complicated structures. Besides the GEV family, we note that the mixed
logit model (MMNL) \citep{McFaTrai00} is  popular due to its flexibility in capturing utility
correlations. In the context of the MCP, the use of the MMNL however yields the same problem structure as the one from the MNL model \citep{MaiLodi2020_OA,Dam2021submodularity}.

Besides the GEV family and MMNL model, the literature has seen other  discrete choice models that would be useful for people demand modeling. For example,
\cite{MarkovC_blanchet2016markov} propose  a \textit{Markov-chain choice model} under which the substitution from one
product to another is modeled as a state transition in a Markov chain. \cite{gallego2019threshold} propose a \textit{threshold utility model} where consumers buy any product whose net utility
exceeds a non-negative, product-specific threshold. \cite{farias2013nonparametric} propose a \textit{non-parametric  choice model} where the choice probabilities are modelled based on  a distribution over all permutations of the choice alternative preferences. 
\citep{MDC_mishra2014theoretical} develop the \textit{marginal distribution choice} (MDC) models based on the assumption that the distribution of  the random utilities are not given exactly but belong to an ambiguity set with marginal information. With techniques from distributionally robust optimization, the estimation of  MDC models in decision-making can be handled by convex optimization \citep{DCM_yan2022representative}.
\mtien{The difference between the MDC  and our robust models lies in the  sources of uncertainty each model captures. More precisely, under the RUM principle,  the utility of an alternative $j$ is modelled as $u_j = v_j+\xi_j$, where $v_j$ is deterministic and $\xi_j$ is assumed to follow a given distribution. While our model captures uncertainties associated with the deterministic term $v_j$, the MDC model assumes that the distribution of $\xi_j$ is ambiguous. In fact, our robust model is  based on the well-studied GEV family while the MDC is not, and our robust model is capable of naturally capturing uncertainties caused by, for instance, estimation errors, limited data, or lack of information about facilities and customers when specifying the utility function. The MDC model, to the best of our knowledge, is limited in this aspect. Moreover, while the application of GEV models in choice-based optimization has had a lot of success, the application of the MDC model has just received attention recently and is still limited. }


In the context of facility location under RUM, there are a number of works making use of the MNL model to capture customers' demand. For example, \cite{BenaHans02} formulate the first MCP under MNL and propose  methods based on mixed-integer linear programming (MILP) and variable
neighborhood search (VNS). Afterward, some
alternative MILP models have been proposed  by \cite{Zhang2012} and \cite{Haase2009}. \cite{Haase2013} then provide a comparison of existing MILP models and conclude that the MILP from \cite{Haase2009} gives the best performance. \cite{Freire2015} strengthen the MILP reformulation of \citep{Haase2009} using  a branch-and-bound algorithm with some tight
inequalities. \cite{Ljubic2018outer} propose a branch-and-cut method combining two types of cutting planes, namely, outer-approximation and submodular cuts, and \cite{MaiLodi2020_OA} propose a multicut outer-approximation algorithm to efficiently solve large instances. All the above papers employ the MNL or MMNL models, leveraging the linear fractional structures of the objective functions to develop solution algorithms. Recently, \cite{Dam2021submodularity} make the first effort to bring general GEV models into the MCP. In this work, we show that the objective function of the MCP under GEV is monotonic submodular, leading to the development of an efficient local search algorithm with a performance guarantee. 

Our work  belongs to the general literature of robust facility location where the problem is  to open new (or reopen available) facilities under uncertainty. For example,  \cite{averbakh1997minimax}
considers a minimax-regret formulation of the
weighted $p$-center problem and shows that 
the problem can be solved by solving a sequence of  deterministic
$p$-center problems. Subsequently, \cite{averbakh2000algorithms} 
study a minimax-regret $1$-center
problem with uncertain node weights and edge lengths and \cite{averbakh2000minmax} consider a minimax-regret $1$-median problem with
interval uncertainty of the nodal demands. 
Some distributionally robust facility location models have  been studied, for instance, \cite{lu2015reliable} study a distributionally robust reliable facility location problem by optimizing over worst-case distributions based on a given distribution of random facility disruptions, \cite{liu2019distributionally} study a distributionally robust model for optimally locating emergency medical service stations under demand uncertainty, and \cite{BeseAlge11} study   a facility location problem where the distribution of customer demand is dependent on location decisions.
We refer the reader to \cite{snyder2006facility_review} for a review. 
Our work differs from the above papers as we employ RUM models, thus customers' behavior is captured by a probabilistic discrete choice model. 

Our work  relates to the rich literature of robust optimization \citep[e.g.][]{ei1997robust,ben1998robust,ben1999robust,bertsimas2004robust} and distributionally robust optimization \citep{wiesemann2014distributionally, rahimian2019distributionally}. Most of the works in the robust/distributionally robust optimization literature are focused on linear or general convex programs, thus the existing methods do not apply to our robust MCP. 
It is worth mentioning some robust models in the  assortment and/or pricing optimization literature, where RUM models are  made use to capture customers' behavior. Some examples are \cite{rusmevichientong2012robust}, \cite{chen2019robust}, and \cite{mai2019robust}, noting that, in the context of assortment and/or pricing optimization, the objective function is often based on one fraction, instead of a sum of fractions as in the context of the MCP. Moreover, the objective function of an assortment problem is typically not submodular, even when the choice model is MNL. \mtien{Our robust MCP  model under MNL closely relates to the robust fractional  0-1 program studied in \cite{mehmanchi2020robust} but  differs by the fact that our approach works with any convex uncertainty sets while the approaches of \cite{mehmanchi2020robust} rely on a particular uncertainty structure introduced by \cite{bertsimas2004robust}. }

\textbf{Paper outline:} We organize the paper as follows. Section \ref{sec:background} provides a background of the deterministic MCP under RUM models. Section \ref{sec:robustMCP} presents our main results for the robust MCP. Section \ref{sec:algo} describes our algorithms used to solve the robust problems. Section \ref{sec:expertiments} provides some numerical experiments, and Section \ref{sec:concl} concludes. Missing proofs  and  additional experiments are provided in the appendix.  

\textbf{Notation:}
Boldface characters represent matrices (or vectors), and $a_i$ denotes the $i$-th element of vector $\ba$. We use $[m]$, for any $m\in \mathbb{N}$, to denote the set $\{1,\ldots,m\}$. 

%
%

\section{Background: Deterministic MCP under RUM}
\label{sec:background}
In this section, we first revisit the RUM framework and then describe the deterministic MCP under RUM models.

\subsection{The RUM Framework and the GEV family}
The RUM framework \citep{McFa78} consists of prominent  discrete choice models for modeling human behavior when faced with a set of discrete choice alternatives. Under the RUM principle, the customers are assumed to associate a random utility $u_j$ with each choice alternative  $j$ in a given choice set of available alternatives $S$. The additive RUM \citep{Fosgerau2013,McFa78} assumes that each random utility is a sum of two parts $u_j = v_j +\xi_j$, where the term $v_j$ is deterministic and can include values representing the  characteristics of the choice alternative $j$ and/or the customers, and $\xi_j$ is random and  unknown to the analyst. Assumptions then can be made for the random terms $\xi_j$, leading to different discrete choice models, i.e., the MNL model relies on the assumption that $\xi_j$ are i.i.d \textit{Extreme Value type I}. The RUM principle then assumes that a choice is made by maximizing the random utilities, thus the probability that an alternative $j$ is selected can be computed as $P (u_j \geq u_k,\:\forall k\in S)$. 

Among RUM models, the MNL model is the simplest one and it is well known that this model fails to capture the correlation between choice alternatives, due to the i.i.d. assumption imposed on the random terms. This drawback is  called as the IIA property \citep{Trai03}. Efforts have been made to relax this property, leading to more advanced choice models with flexible correlation structures such as the nested logit or cross-nested logit models.  Among them, the GEV family \citep{McFa81,Daly2006} is regarded as one of the most general families of discrete choice models in the econometrics and operations research literature. We describe this family of models in the following. 

Assume that the choice set contains $m$ alternatives indexed as $\{1,\ldots,m\}$ and let $\{v_1,\ldots,v_m\}$ be the vector of deterministic  utilities of the $m$ alternatives. 
A GEV model can be represented by a choice probability generating function (CPGF) $G(\bY)$ \citep{McFa81,Fosgerau2013}, where $\bY$ is a vector of size $m$ with entries $Y_j = e^{v_j}$, $j\in[m]$. Given $j_1,\ldots, j_k \in  [m]$, let $\partial G_{j_1...j_k}$,  be the mixed partial derivatives of $G$ with respect to $Y_{j_1},\ldots,Y_{j_k}$.  
The following basic properties hold for any CPGF  $G(\cdot)$ in the GEV family \citep{McFa78}. 
\begin{remark}[Basic properties of GEV's CPGF]
\label{prp:CPGF}
\textit{The following properties hold for any GEV probability generating function.
\begin{itemize}
 \item[(i)] $G(\bY) \geq 0,\ \forall \bY\in \bbR^m_+$,
 \item[(ii)] $G(\bY)$ is homogeneous of degree one, i.e., $G(\lambda \bY) = \lambda G(\bY)$, for any scalar $\lambda>0$
 \item[(iii)] $G(\bY)\rightarrow \infty$ if $Y_j\rightarrow \infty$, {for any $j\in [m]$}
 \item[(iv)]  Given $j_1,\ldots,j_k \in [m]$ distinct from each other,
  $\partial G_{j_1,\ldots,j_k}(\bY)>0$ 
 if $k$ is odd, and \mtien{$\leq 0$} if $k$ is even.
\end{itemize}}
\end{remark}
The above properties are standard for the GEV family. \mtien{The economic intuition behind these properties is however limited \citep[Section 4.6 in][]{Trai03}, especially for Property (iv).} In fact, these properties (or conditions) are to ensure that the corresponding choice model is consistent with the RUM principle and would be useful to design a new \mtien{RUM model.}
\mtien{To support our later exposition, we  present some additional properties of the GEV family in Proposition \ref{prp:CPGF-new} below.}
These new properties can be verified  easily using the  basic properties introduced above.  
\begin{proposition}[Some additional properties of GEV's CPGF] 
\label{prp:CPGF-new}
The following properties  hold for any CPGF under the GEV family: 
\begin{itemize}
    \item[(i)] $G(\bY) = \sum_{j\in [m]} Y_j\partial G_j(\bY)$
 \item[(ii)] $\partial G_j(\lambda \bY) = \partial G_j( \bY)$ for any scalar $\lambda>0$
 \item[(iii)] $\sum_{k\in [m]} Y_k\partial G_{jk} (\bY) = 0$, $\forall j\in [m]$.
\end{itemize} 
\end{proposition}
\begin{proof}
{Property} (i) can be obtained by taking the derivatives with respect to $\lambda$ on both sides of $G(\lambda \bY) = \lambda G(\bY)$ (Property (ii) of Remark \ref{prp:CPGF}) to have $
G(\bY) = \sum_{j\in [m]}Y_j \partial G_j(\lambda \bY)$. We then  let $\lambda = 1$ to obtain the desired equality. 

For Property (ii), we take derivatives  on both sides of the equality $G(\lambda \bY) = \lambda G(\bY)$ with respect to $Y_j$ to have
$$\lambda \partial G_j(\lambda \bY) = \lambda{\partial} G_j(\bY).$$
By removing $\lambda$ from both sides of the equality, we can  obtain the desired equality $\partial G_j(\lambda \bY) = {\partial} G_j(\bY)$. 

For Property (iii),  we further take the derivatives with respect to $\lambda$ on both sides of (ii), we get $
\sum_{k\in [m]}Y_k\partial G_{jk}(\lambda \bY) = 0$. By letting $\lambda = 1$, we  obtain (iii).
\end{proof}

Under a GEV model specified by a CPGF $G(\bY)$, the choice probabilities are given as  
\[
P(j|\bY,G)  = \frac{Y_i \partial G_i(\bY)}{G(\bY)}.
\]
If the choice model is MNL, the CPGF has a linear form $G(\bY)=\sum_{j\in [m]}Y_j$ and the corresponding choice probabilities are given by a linear fraction  $P(j|\bY,G) = {Y_j}/({\sum_{j\in [m]} Y_j})$. On the other hand, if the choice model is a nested logit model, 
 the choice set can be partitioned into $L$ nests, which are disjoint subsets of the alternatives. Let denote by $n_1,\ldots,n_L$ the $L$ nests.
The corresponding CPGF  has a nonlinear form as 
$
G(\bY) = \sum_{l\in L} \left(\sum_{j\in n_l} Y_j^{\mu_l} \right)^{1/\mu_l},
$
where  $\mu_l\geq 1$, $l\in[L]$, are the parameters of the nested logit model. The choice probabilities  have the more complicated form
\[
P(j|\bY,G) = \frac{\left(\sum_{j'\in n_l} Y_{j'}^{\mu_l}\right)^{1/\mu_l}}{\sum_{l\in [L]}\left(\sum_{j'\in n_l} Y_{j'}^{\mu_l}\right)^{1/\mu_l}} \frac{Y_j^{\mu_l}}{\sum_{j'\in n_l} Y_{j'}^{\mu_l}},\; \forall l\in [L], j\in n_l.  
\]
In a more general setting, a CPGF can be represented by a rooted network \citep{Daly2006}, for which the choice probabilities may have no closed-form and need to be computed by recursion or dynamic programming \citep{MaiFreFosBas15_DynMEV}. 

\subsection{The Deterministic MCP}
In the context of the MCP under  a RUM model, a firm would like to select some locations from a set of available locations to set up new facilities, assuming that there exist facilities from the competitor in the market. The firm then aims to maximize an expected market share achieved by attracting customers to the new facilities. 
We suppose that there are $m$ available locations and we let  $[m] = \{1,2,\ldots,m\}$ be the set of  locations. Let $I$ be the set of geographical zones where customers are located and $q_{i}$ be the number of customers in zone $i\in I$. For each customer zone $i$, let $v_{ij}$ be the corresponding deterministic utility of location $j\in [m]$. These utility values can be inferred by estimating the discrete choice model using observational data of how customers made decisions. 
For each customer zone $i \in I$, the corresponding  discrete choice model  can be represented by a CPGF $G^i(\bY^i)$, where $\bY^i$ is a vector of size $m$ with entries $Y^i_j = e^{v_{ij}}$. 
The choice probability  of a location $j\in [m]$ can then be computed as
\[
P(j|\bY^i) = \frac{Y_j \partial G^i_j(\bY^i)}{U^i+ G^i(\bY^i)},
\]
where $U^i$ is the total utility  of the competitor for zone $i\in I$.
Such an utility is analogous to a non-purchase utility used in the context of assortment  optimization \citep{Talluri2004revenue}.
Note that we can write
\begin{equation*}
P(j|\bY^i) =\frac{\frac{Y_j}{U^i} \partial G^i_j(\bY^i)}{1 + \frac{1}{U^i}G^i(\bY^i)}   \stackrel{(a)}{=} \frac{\frac{Y_j}{U^i} \partial G^i_j\left(\frac{\bY^i}{U^i}\right)}{1 + G^i\left(\frac{\bY^i}{U^i}\right)},
\end{equation*}
where  \textit{(a)} is due to Property (ii) of Proposition \ref{prp:CPGF}. Thus, $P(j|\bY^i)$ are   the choice probabilities given by the  set of  utilities $v'_{ij} =v_{ij} - \ln U^i$. In other words, one can subtract the utilities $v_{ij}$ by $\ln U^i$ to force the competitor's utilities to be 1, without  any loss of generality. Therefore, from now on, for the sake of simplicity, we assume $U^i = 1$, for all $i\in I$.

In the context of the MCP, we are interested in choosing a subset of locations $S\subset [m]$ to locate new facilities. 
Hence, the conditional choice probabilities of choosing a location $j \in S$  can be written as
\[
P(j|\bY^i,S) = \frac{Y^i_j \partial G^i_j(\bY^i|S)}{1+ G^i(\bY^i|S)},\; \forall j\in S,
\]
where the $G^i(\bY^i|S)$ is defined as $G^i(\bY^i|S) = G^i(\overline{\bY}^i)$, where $\overline{\bY}^i$ is a vector of size $m$ with entries $\overline{Y}^i_j = {Y}^i_j$ if $j\in S$ and $\overline{Y}^i_j =0$ otherwise. This can be interpreted as if a location $j$ is not selected, then its utility should be  very low, i.e., $v_{ij} = -\infty$, then $Y^i_j = e^{v_{ij}} = 0$.  
The deterministic MCP under GEV models specified by CPGFs $G^i(\bY^i)$, $i\in I$, can be defined as  
\begin{equation}\label{prob:MCP-1}\tag{\sf MCP}
 \max_{S \in \cK}\left\{f(S) = \sum_{i\in I}q_i\sum_{j\in S} P(j|\bY^i, S) \right\},
\end{equation}
where $\cK$ is the set of feasible solutions. Under a conventional cardinality constraint $|S| \leq C$, $\cK$ can be defined as $\cK = \{S\subset[m]|\; |S|\leq C\}$, for a given constant $C$ such that $1\leq C\leq m$.
\eqref{prob:MCP-1} is generally NP-hard, even under the MNL choice model. However, it is possible to obtain $(1-1/e)$ approximation solutions using a greedy local search algorithm \citep{Dam2021submodularity}.

Under the MNL model, the MCP can be formulated as a linear fractional  program as
\[
\max_{S \in \cK}\left\{f(S) =  \sum_{i\in I} q_i \sum_{i\in I}  \frac{\sum_{j\in S} Y^i_j}{ 1+ \sum_{j\in S} Y^i_j}  \right\}.
\]
The fractional structure allows to formulate the MCP under MNL as a MILP, thus a MILP solver, e.g., CPLEX or GUROBI, can be used \citep{Haase2009,Haase2013,Freire2015}. It is  well-known that the objective function $f(S)$ is submodular, leading to some approaches based on submodular cuts or local search heuristic \citep{Ljubic2018outer,Dam2021submodularity}. Under more general GEV models, e.g., the nested logit or cross-nested logit models \citep{Trai03}, the objective function becomes much more complicated. In fact, it would be possible to reformulate the MCP under a GEV model as a MILP with submodular cuts  \citep{Nemhauser1978analysis,Benati2002,Ljubic2018outer}. \mtien{However, such a MILP approach involves an exponential number of constraints, thus would be not  tractable to solve. In addition, under a general GEV model, the choice probabilities would not be expressed in a closed form \citep{mai2017dynamic}, making the computation of $f(S)$ intractable and the corresponding MILP intractable as well. }

It is \mtien{relevant} to connect the MCP formulation to the context of assortment optimization under discrete choice models. In fact, \eqref{prob:MCP-1} under MNL  shares a close structure   with assortment optimization problems under the mixed logit model \citep{rusmevichientong2014assortment}. \mtien{However,  \eqref{prob:MCP-1} is  more tractable, in the sense it can be written as a binary program with a convex objective function, allowing for some convex optimization techniques (e.g. the outer-approximation algorithm) to be applied , while it is not the case in  assortment optimization. 
Moreover, under the GEV family, the  assortment problem becomes even more challenging to solve and, to the best of our knowledge, there is no algorithm with performance guarantees for such a problem. On the other hand, one can achieve $(1-1/e)$ approximation solutions to the MCP by just using a simple greedy heuristic \citep{Dam2021submodularity}.}

\mtien{There would be relevant situations where  
the algorithms developed in this paper can apply to assortment optimization problems. For example,  one can  think of a situation where a seller needs to select a set of products to sell together with a competitor, and the objective is to maximize the expected number of customers that come to purchase their products, instead of maximizing an expected revenue in a conventional assortment optimization problem. In this context, the assortment optimization model shares the same structure with \eqref{prob:MCP-1} and the  methods developed in this paper can apply.}

\section{Robust MCP}
\label{sec:robustMCP}
We first consider the robust MCP under any GEV model and show that the robust model preserves the monotonicity and submodularity from the deterministic one, which will ensure that a simple greedy heuristic will always return a $(1-1/e)$ approximation solution. Moreover, we show that under the MNL model, the robust model  preserves the concavity, implying that an outer-approximation algorithm could be used to efficiently  solve the robust MCP to optimality. 

\subsection{Robust MCP under GEV models}
In our uncertainty setting, it is assumed that the vectors of utilities are not known with certainty  but belong to some uncertainty sets. That is, we assume that for each customer zone $i\in I$, the corresponding deterministic utilities $\bv^i = \{v_{ij}|\; j\in [m]\}$ can vary in an uncertainty set $\cV_i$ 
and these uncertainty sets are independent across $i\in I$. This setting is natural in the context, in the sense that there is one discrete choice model with choice utilities $\bv^i$ for customers in each zone $i$, and these utilities are typically inferred from observations of how people in that zone made choices. Uncertainty sets could be constructed from this, leading to independent uncertainty sets over customer zones. Moreover, if we relax this assumption, i.e., $\cV_i$ are no longer independent over $i\in I$, the adversary's minimization problem \footnote{When we say ``adversary'', we refer to  the worst-case minimization problem, not the competitor in the market.} will become much more difficult to solve, even under the classical MNL model. In contrast, under the assumption that the uncertainty sets are separable by zones, we will show later that the adversary's problem can be efficiently handled by convex optimization.

We further assume that   $\cV_i$ are convex and bounded for all $i\in I$.
Convexity and boundedness are typical assumptions in the robust optimization literature \citep{Bertsimas2011theoryRO}. 
For later investigations,  we assume that the uncertainty set $\cV_i$ can be defined by a set of constraints $\{g^i_t(\bv^i) \leq 0;\; t = 1,\ldots,T\}$ where  $g^i_t(\bv^i)$ are convex functions in $\bv^i$.
When the choice parameters $\bv^i$, $i\in I$, cannot be identified exactly, we are interested in the worst-case scenario.
The robust version of the classical MCP then can be formulated as 
\begin{equation}
 \max_{S \in \cK}\;\min_{\bv^i \in \cV_i}\left\{f(S,\bV) = \sum_{i\in I}q_i\sum_{j\in S} P(j|\bv^i, S) \right\},    
\end{equation}
where $P(j|\bv^i, S)$ is the choice probability of location $j\in [m]$ given utilities $\bv^i$ and set of locations $S \in \cK$. Under a GEV choice model with CPGFs $G^i$, $i\in I$, we write the choice probabilities as
\[
P(j|\bv^i, S) = \frac{Y^i_j \partial G^i_j(\bY(\bv^i)|S)}{1+ G^i(\bY(\bv^i)|S)},
\]
where $\bY(\bv^i)$ is a vector of size $m$ with entries $Y(\bv^i)_j = e^{v_{ij}}$. The objective function can be further simplified as 
\begin{align}
f(S,\bV) &= \sum_{i\in I}q_i\sum_{j\in S} P(j|\bv^i, S) \nonumber \\
&= \sum_{i\in I}q_i \frac{\sum_{j\in S} Y^i_j \partial G^i_j(\bY(\bv^i)|S)}{1+ G^i(\bY(\bv^i)|S)} \nonumber \\
&\stackrel{(a)}{=} \sum_{i\in I} q_i -  \sum_{i\in I}\frac{q_i}{1+G^i(\bY(\bv^i)|S)},
\end{align}
where $(a)$ is due to Property (i) of Proposition \ref{prp:CPGF-new}, i.e., $\sum_{j\in S} Y^i_j \partial G^i_j(\bY(\bv^i)|S) = G^i(\bY(\bv^i)|S)$.
Thus, the robust problem can be reformulated as
\begin{equation}
\label{prb:robust-MCP} \tag{\sf RMCP}
 \max_{S \in \cK} \left\{f^{\WC}(S) = \sum_{i\in I} q_i -  \sum_{i\in I}\frac{q_i}{1+ \min_{\bv^i \in \cV_i} \{G^i(\bY(\bv^i)|S)\} } \right\},    
\end{equation}
where $f^{\WC}(S)$ is referred to as the worst-case objective function when the choice parameters of the GEV model vary in the uncertainty sets.

\mtien{In  \eqref{prb:robust-MCP} we assume that there is no uncertainty associated with the specification of the competitor's utilities. In a more general setting, it is possible that the competitor's utilities are not known with certainty and would need to be taken into consideration in the robust model. Nevertheless, we will show, in the following,  that such a general uncertainty structure can be converted into the same uncertainty structure in   \eqref{prb:robust-MCP} with shifted (convex) uncertainty sets. To facilitate our exposition,  let us assume that the utility of the competitor is $v^i_0$ for customer zone $i\in I$, and both $v^i_0, \bv^i$ can vary in an uncertainty set $\cV_i$. The robust problem then becomes:}
\begin{equation}
\label{prob:general-1}
 \max_{S \in \cK}\;\min_{\substack{(v^i_0,\bv^i) \in \cV_i\\ \forall i\in I}}\left\{ \sum_{i\in I}q_i \frac{\sum_{j\in [m]}Y^i_j(v^i_j) \partial G^i_j(\bY(\bv^i)|S)}{e^{v^i_0}+ G^i(\bY(\bv^i)|S)} \right\}.    
\end{equation}
Proposition \ref{prp:equivalent-general} shows that \eqref{prob:general-1} can be converted equivalently into a \mtien{robust} problem of the same structure as \eqref{prb:robust-MCP}. The proof can be found in Appendix \ref{appd:proofs}.
\begin{proposition}
\label{prp:equivalent-general}
\eqref{prb:robust-MCP} is equivalent to
\begin{equation}
\label{prob:general-2}
 \max_{S \in \cK}\;\min_{\substack{\widetilde{\bv}^i \in \widetilde{\cV}_i\\ \forall i\in I}}\left\{ \sum_{i\in I}q_i \frac{\sum_{j\in [m]}Y^i_j(\widetilde{v}^i_j) \partial G^i_j(\bY(\widetilde{\bv}^i)|S)}{1 + G^i(\bY(\widetilde{\bv}^i)|S)} \right\},  
\end{equation}
where $\widetilde{\cV}_i = \left\{ \widetilde{\bv}^i\in \bbR^{m}\Big|\; \exists \bv^i\in \cV_i \text{ s.t. } \widetilde{v}^i_j = v^i_j - v^i_0,\;\forall j\in [m]\right\}.$
\end{proposition}

It can be seen that $\widetilde{\cV}_i$ is convex if $\cV_i$ is convex. The inner minimization problem of \eqref{prob:general-1} differs from the inner minimization of \eqref{prb:robust-MCP} just by some  additional linear constraints, i.e., 
\begin{align}
    \underset{}{\min} \qquad &  \sum_{i\in I}q_i \frac{\sum_{j\in [m]}Y^i_j(\widetilde{v}^i_j) \partial G^i_j(\bY(\widetilde{\bv}^i)|S)}{1 + G^i(\bY(\widetilde{\bv}^i)|S)}  \nonumber\\
    \text{subject to} \qquad &  \widetilde{v}^i_j = v^i_j-v^i_0 &\forall i\in I, j\in [m] \nonumber\\
    & (v^0_j,\bv^i_j) \in \cV_i &\forall i\in I \nonumber.
\end{align}
\mtien{Thanks to Proposition \ref{prp:equivalent-general},  for the sake of simplicity, we will keep the  assumption that the utilities of the competitor are deterministically equal to 1 throughout the rest of the paper.}

We now explore some properties of the robust program in the following.
\mtien{Proposition \ref{prp:convexity-G} below first shows} that the inner minimization problem (adversary's problem) can be solved efficiently via convex optimization.  
Before discussing the result, let us define a binary representation of $f^{\WC}(S)$ as follows.  For any set $S\in\cK$, let $\bx^S \in \{0,1\}^m$ such that $x^S_j = 1$ if $j\in S$ and $x^S_j = 0$ otherwise. Then, we can write $G^i(\bY(\bv^i)|S) = G^i(\bY(\bv^i) \circ \bx^S)$, where $\circ$ is the element-by-element operator. 
\begin{proposition}[Convexity of the adversary's minimization problems]
\label{prp:convexity-G}
Given any $i\in I$ and $S\subset [m]$, $G^i(\bY(\bv^i)|S)$ is strictly convex in $\bv^i$. 
\end{proposition}

As shown in \cite{Dam2021submodularity}, the objective function of the deterministic MCP is monotonic increasing, i.e., adding more facilities always yields better objective values. The following proposition shows that the robust model preserves the monotonicity. 
\begin{theorem}[Robustness Preserves the Monotonicity]
\label{thm:monotonicity}
Given $S\subset [m]$, for any $j\in [m]\backslash S$,  we have $$f^{\WC}(S \cup \{j\}) > f^{\WC}(S).$$ 
\end{theorem}
\begin{proof}
Let $\bx \in \{0,1\}^m$ be the binary presentation of  set $S \subset [m]$. We write the objective function as 
\[
f^{\WC}(\bx) =  \sum_{i\in I} q_i -  \sum_{i\in I}\frac{q_i}{1+ \min_{\bv^i \in \cV_i} \{G^i(\bY(\bv^i) \circ \bx )\} }
\]
We need to prove that, for any $\bx \in X $  and $j\in [m]$ such that $x_j = 0$, we have $f^{\WC}(\bx + \bbe^j) > f^{\WC}(\bx)$, where $e^{j}$ is a vector of size $m$ with all zero elements except the $j$-th one which is equal to 1. To prove this, let us consider $G^i(\bY(\bv^i) \circ \bx)$. Taking the derivative of $G^i(\bY(\bv^i) \circ \bx)$ w.r.t. $x_j$ we have
\[
\frac{\partial G^i(\bY(\bv^i) \circ \bx)}{\partial x_j} = { Y^i_j(\bv^i)\partial G^i_j(\bY(\bv^i) \circ \bx)} \stackrel{(b)}{>}0,  
\]
where the (strict) inequality $(b)$ is because $Y^i_j(\bv^i) = e^{v_{ij}}>0$ and $\partial G^i_j(\bY(\bv^i) \circ \bx)>0$ (Property (iv) of Remark \ref{prp:CPGF}). 
Thus, 
\[
G^i(\bY(\bv^i) \circ (\bx+\bbe^j)) > G^i(\bY(\bv^i) \circ \bx),
\]
and consequently,
\[
\min_{\bv^i \in\cV_i} \{G^i(\bY(\bv^i) \circ (\bx+\bbe^j))\} > \min_{\bv^i \in\cV_i} \{G^i(\bY(\bv^i) \circ \bx)\}
\]
which directly leads to the desired inequality $f^{\WC}(\bx + \bbe^j) > f^{\WC}(\bx)$. 
\end{proof}

Similar to the deterministic case, the monotonicity implies that an optimal solution $S^*$ to \eqref{prb:robust-MCP} always reaches its maximum capacity, i.e., $|S^*| = C$. 

It is known  that the objective function of the deterministic MCP is submodular \citep{Dam2021submodularity}. 
Typically,  a robust version of a monotonic submodular function is not submodular \citep{Orlin2018robustSubModular}.
However, in the theorem below, we show that the inclusion of the adversary in our robust MCP  preserves the submodularity. 
\begin{theorem}[Robustness Preserves the Submodularity]
\label{thm:submodularity}
$f^{\WC}(S)$ is submodular, i.e., for two sets $A\subset B \subset [m]$ and for any $j\in [m]\backslash  B$ we have
\[
f^{\WC}(A \cup \{j\}) - f^{\WC}(A) \geq   f^{\WC}(B \cup \{j\}) - f^{\WC}(B).
\]
\end{theorem}
In the following we provide essential steps to prove the submodularity result.
We will employ the binary representation of $f^{\WC}(S)$ to prove the claim. Let 
\begin{equation}
\label{eq:define-phi}
\phi^i(\bx)  = \min_{\bv^i\in\cV_i} G^i(\bY(\bv^i) \circ \bx).    
\end{equation}
Since $G^i(\bY(\bv^i) \circ \bx)$ is strictly convex and differentiable, $\phi^i(\bx)$ is  continuous and differentiable  \citep[Corollary 8.2 in ][]{Hogan1973}.  
We first introduce the following lemma.
\begin{lemma}[First and second order derivatives of $\phi^i(\bx)$]
\label{lm:lm1}
Given $\bx \in [0,1]^m$, let $\bv^{i*}$ be an optimal solution to the problem  $\min_{\bv^i\in\cV_i} G^i(\bY(\bv^i) \circ \bx)$, then for any $j,k\in [m]$, 
\begin{align}
    \frac{\partial \phi^i(\bx)}{\partial x_j} &= \frac{\partial G^i(\bY(\bv^{i*}) \circ \bx)}{\partial x_j}\nonumber\\
    \frac{\partial^2 \phi^i(\bx)}{\partial x_j \partial x_k} &= \frac{\partial^2 G^i(\bY(\bv^{i*}) \circ \bx)}{\partial x_j \partial x_k}.\nonumber
\end{align}
\end{lemma}

Note that $\min_{\bv^i\in\cV_i} G^i(\bY(\bv^i) \circ \bx)$ always yields a unique optimal solution due to the strict concavity of $G^i(\bY(\bv^i) \circ \bx)$ (Proposition \ref{prp:convexity-G}). Thus, $\bv^{i*}$ \mtien{mentioned} in Lemma \ref{lm:lm1} is always unique.

We need one more lemma to complete the proof of the submodularity result. The following lemma shows a monotonicity  behavior of ${\partial \phi^i(\bx)}/{\partial x_k}$ as a function of $\bx$. 
\begin{lemma}\label{lm:lm2}
Given $\bx \in \{0,1\}^m$ and any $j,k\in [m]$ such that $x_j = x_k =0$ and $j\neq k$,
we have
\[
\frac{\partial \phi^i(\bx + \bbe^j)}{\partial x_k} \leq \frac{\partial \phi^i(\bx)}{\partial x_k}.
\]
\end{lemma}
\begin{proof}
We define
$
\psi(\bx) = \partial \phi^i(\bx)/\partial x_k.
$
Taking the first-order derivative of $\psi(\bx)$ w.r.t. $x_j$ we get
\begin{align}
\frac{\partial \psi(\bx)}{\partial x_j} &= \frac{\partial^2 \phi^i(\bx)}{\partial x_k \partial x_j}    \nonumber \\
&\stackrel{(a)}{=} \frac{\partial^2 G^i(\bY(\bv^{i*}) \circ \bx)}{\partial x_j \partial x_k}\nonumber \\
&\stackrel{}{=} Y_j(\bv^{i*}) Y_k(\bv^{i*}) \partial^2G^i_{jk}(\bY(\bv^{i*}) \circ \bx) \stackrel{(b)}{\leq} 0
\end{align}
where $(a)$ is from Lemma \ref{lm:lm1} and  $(b)$ is due to Property (iv) of Remark \ref{prp:CPGF}. So, we have ${\partial \psi(\bx)}/{\partial x_j} \leq 0$, implying that $\psi(\bx)$ is monotonically decreasing in $x_j$. Thus, ${\partial \phi^i(\bx + \bbe^j)}/{\partial x_k} \leq {\partial \phi^i(\bx)}/{\partial x_k}$, as desired.
\end{proof}

We are now ready for the main proof of Theorem
\ref{thm:submodularity}. 

\begin{proof}[Proof of Theorem \ref{thm:submodularity}]
From Theorem \ref{thm:monotonicity}, we have
\begin{align}
\phi^i(\bx+\bbe^j) &\geq  \phi^i(\bx) \label{eq:proof-th3-eq01}\\
\phi^i(\bx+\bbe^j + \bbe^k) &\geq \phi^i(\bx+\bbe^k ). \label{eq:proof-th3-eq02}
\end{align}
Moreover,  from Lemma \ref{lm:lm2}
we see that, for any $\bx\in \{0,1\}^m$ and any $j,k\in [m]$ such that $x_j=x_k = 0$,  function $\psi(\bx) = \phi^i(\bx+\bbe^j) - \phi^i(\bx)$ is monotonically decreasing in $x_k$. Thus,
\begin{equation}
\label{eq:proof-th3-eq1}
\phi^i(\bx+\bbe^j) - \phi^i(\bx) \geq \phi^i(\bx+\bbe^j + \bbe^k) - \phi^i(\bx + \bbe^k)\geq 0. 
\end{equation}
Moreover, from \eqref{eq:proof-th3-eq01} and \eqref{eq:proof-th3-eq02}, we  have
\begin{equation}
\label{eq:proof-th3-eq2}   
(1+\phi^i(\bx+\bbe^j))(1+\phi^i(\bx)) \leq (1+ \phi^i(\bx+\bbe^j + \bbe^k))(1+\phi^i(\bx + \bbe^k)).
\end{equation}
Combine \eqref{eq:proof-th3-eq1} and \eqref{eq:proof-th3-eq2} we get
\begin{align}
&\frac{\phi^i(\bx+\bbe^j) - \phi^i(\bx)}{(1+\phi^i(\bx+\bbe^j))(1+\phi^i(\bx))} \geq \frac{\phi^i(\bx+\bbe^j + \bbe^k) - \phi^i(\bx + \bbe^k)}{(1+ \phi^i(\bx+\bbe^j + \bbe^k))(1+\phi^i(\bx + \bbe^k))}\nonumber \\
&\Leftrightarrow \frac{1}{1+\phi^i(\bx)} - \frac{1}{1+\phi^i(\bx+ \bbe^j)} \geq \frac{1}{1+ \phi^i(\bx+\bbe^k )} - \frac{1}{1+ \phi^i(\bx+\bbe^j +\bbe^k )}. \label{eq:th1-eq1}
\end{align}
Now we note that
\[
f^{\WC}(\bx) = \sum_{i\in I} q_i -   \sum_{i\in I}\frac{q_i}{1+ \phi^i(\bx) }.
\]
Thus, from \eqref{eq:th1-eq1} we have
\begin{equation}
\label{eq:th1-eq2}
 f^{\WC}(\bx+\bbe^j) - f^{\WC}(\bx) \geq  f^{\WC}(\bx+\bbe^j +\bbe^k) - f^{\WC}(\bx + \bbe^j). 
\end{equation}
Now, given  $A\subset B\subset [m]$, let $\bx^A$ and $\bx^B$ be the binary representations of $A$, $B$, respectively. From \eqref{eq:th1-eq2} we have, for any $j \in [m]\backslash B$
\[
\begin{aligned}
 f^{\WC}(\bx^A+\bbe^j) - f^{\WC}(\bx^A) &\geq f^{\WC}\left(\bx^A+\bbe^j +\sum_{k\in B\backslash A} \bbe^k\right) - f^{\WC}\left(\bx^A  +\sum_{k\in B\backslash A} \bbe^k\right)\\
 &= f^{\WC}(\bx^B+\bbe^j) -  f^{\WC}(\bx^B),
\end{aligned}
\]
which is equivalent to $f^{\WC}(A \cup \{j\}) - f^{\WC}(A) \geq   f^{\WC}(B \cup \{j\}) - f^{\WC}(B)$, implying the submodularity as desired.  
\end{proof}

\mtien{The submodularity and monotonicity  imply that,} under a cardinality  constraint $|S|\leq C$,  a greedy heuristic can guarantee a $(1-1/e)$ approximation solution to the robust problem (Corollary \ref{coro:co1} below) \citep{Nemhauser1978analysis}. Such a greedy heuristic can  start from an empty set and keep adding locations, one at a time, taking at each step a location that increases the worst-case objective function $f^{\WC}(S)$ the most. The algorithm stops when the maximum capacity is reached,  i.e.,  $|S| = C$. This greedy algorithm runs in $\cO(mC)$.  

\begin{corollary}[Performance guarantee for a greedy heuristic]
\label{coro:co1}
For the robust MCP under a cardinality constraint $|S|\leq C$,  there exists a greedy heuristic that always returns a solution $S^*$ such that
\[
f^{\WC}(S^*) \geq (1-1/e) \max_{|S| \leq C} \left\{ f^{\WC}(S)\right\}.
\]
\end{corollary}

If the uncertainty set is rectangular, i.e., each $v^i_j$ can freely deviate within an interval, the following proposition shows that  the robust MCP can be further converted into a deterministic MCP.
\begin{proposition}[Rectangular uncertainty  sets]
If the uncertainty sets are rectangular, i.e., $\cV_i = \{\bv^i|\; \underline{\bv}^i \leq \bv^i \leq \overline{\bv}^i\}$, for all $i\in I$, then the robust problem \eqref{prb:robust-MCP} is equivalent to
\[
 \max_{S \in \cK} \left\{ \sum_{i\in I} q_i -  \sum_{i\in I}\frac{q_i}{1+ G^i(\bY(\underline{\bv}^i)|S) } \right\},    
\]
 which is a deterministic MCP with parameters $\bv^i = \underline{\bv}^i$, $\forall i\in I$. 
\end{proposition}
\begin{proof}
We will simply show that $G^i(\bY({\bv}^i)|S)$ is minimized over $\bv^i\in\cV_i$ at $\bv^i = \underline{\bv}^i$. To prove this,  we  consider the equivalent binary representation  $G^i(\bY({\bv}^i)\circ \bx^S)$  and take its derivatives w.r.t. $v^i_j$, for any $j\in [m]$, to have 
\[
\frac{G^i(\bY({\bv}^i)\circ \bx^S)}{\partial v^i_{j}} = Y_j(\bv^i) x_j \partial G^i_j(\bY({\bv}^i)\circ \bx^S),
\]
and see that the right hand side is non-negative, because $Y_j(\bv^i)\geq 0$, $x_j\geq 0$ and  $\partial G^i_j(\bY({\bv}^i))\geq 0$, where the later is due to Property (iv) of Remark \ref{prp:CPGF}. So, $G^i(\bY({\bv}^i)\circ \bx^S)$ is monotonically increasing in $\bv^i$. Thus,
$$G^i(\bY({\bv}^i)\circ \bx^S) \geq  G^i(\bY(\underline{\bv}^i)\circ \bx^S),\;\forall \bv^i\in\cV_i, $$
implying that $G^i(\bY({\bv}^i)|S)$ is minimized  at $\bv^i = \underline{\bv}^i$ as desired.
\end{proof}


\subsection{Robust MCP under MNL}
Under the MNL choice model, the deterministic MCP  has a concave objective function \citep{BenaHans02}, making it solvable by an exact method such as  the outer-approximation algorithms \citep{Ljubic2018outer, MaiLodi2020_OA}. In the following, we show that it is  the case for our robust model. First, recall that under the MNL model, the CPGF becomes 
\[
\begin{aligned}
G^i(Y(\bv^i)|S) &= \sum_{j\in S} e^{v_{ij}},\text{ and } 
G^i(Y(\bv^i)\circ \bx) = \sum_{j\in [m]} e^{v_{ij}} x_j. \\
\end{aligned}
\]

\begin{proposition}[Concavity under the MNL model]
\label{prp:concave-MNL}
 $f^{\WC}(\bx)$ is concave in $\bx$. 
\end{proposition}
 The claim can be obviously verified based on the fact that minimization will preserve the concavity. That is, if we define (for ease of notation)
\[
\gamma(\bx, \bV) = \sum_{i\in I} q_i -  \sum_{i\in I}\frac{q_i}{1+  G^i(\bY(\bv^i) \circ \bx ) },
\]
then we know that $\gamma(\bx,\bV)$ is concave in $\bx$ \citep{BenaHans02}.
Note that $f^{\WC}(\bx) = \min_{\bv^i \in \bV,\forall i} \gamma(\bx,\bV)$.
To prove the concavity of $f^{\WC}(\bx)$ we will show that for any $\bx^1,\bx^2\in [0,1]^m$ and $\alpha \in[0,1]$, we have $\alpha f^{\WC}(\bx^1) + (1-\alpha) f^{\WC}(\bx^2) \leq  f^{\WC}(\alpha \bx^1 + (1-\alpha)\bx^2)$. This is verified by the following chain of inequalities
\begin{align}
    \alpha f^{\WC}(\bx^1) + (1-\alpha) f^{\WC}(\bx^2)  &= \min_{\bv^i\in \cV_i,\forall i} \left\{\alpha  \gamma(\bx^1, \bV)\right\} + \min_{\bv^i\in \cV_i,\forall i} \left\{(1-\alpha)  \gamma(\bx^2, \bV)\right\} \nonumber\\
    &\leq \min_{\bv^i\in \cV_i,\forall i} \left\{\alpha  \gamma(\bx^1, \bV) + (1-\alpha)  \gamma(\bx^2, \bV)\right\} \nonumber \\
    &\stackrel{(a)}{\leq } \min_{\bv^i\in \cV_i,\forall i} \left\{  \gamma(\alpha \bx^1 + (1-\alpha)\bx^2, \bV)\right\}\nonumber\\
    & = f^{\WC}(\alpha \bx^1 + (1-\alpha)\bx^2)\nonumber
\end{align}
where $(a)$ is due to the concavity of  $\gamma(\bx,\bV)$.

Note that the concavity is preserved with any uncertainty set, not necessarily with convex and customer-wise decomposable sets.  That is, under any nonempty uncertainty set $\cV$ such that the worst-case objective function $$f^{\WC}(\bx) = \min_{\{\bv^1,\ldots,\bv^{|I|}\} \in \cV}  \left\{\sum_{i\in I} q_i -  \sum_{i\in I}\frac{q_i}{1+  G^i(\bY(\bv^i) \circ \bx ) }\right\} $$ is finite, then $f^{\WC}(\bx)$ is  concave in $\bx$. However, under this general setting,  the above minimization problem is difficult to solve, as its objective function is highly non-convex in $\bV = (\bv^i,\;i\in I)$.   

\section{Robust Algorithms}
\label{sec:algo}
For the MCP  under GEV models, \cite{Dam2021submodularity} proposes a local search procedure, named as GGX, which does not only provide a performance-guaranteed solution but  performs  well in practice. On the other hand, if the choice model is MNL,  it is possible to efficiently solve the deterministic MCP  using a multi-cut outer-approximation algorithm \citep{MaiLodi2020_OA}. As shown above, our robust models preserve some main properties of the deterministic versions, i.e., the submodularity and monotonicity for the MCP under GEV models,  and concavity for the MCP under MNL, making the local search and outer-approximation approaches still useful. In the following, we discuss how such approaches can be adapted to handle the robust MCP.   

\subsection{ Local Search for GEV-based Robust MCP}
The submodularity  and monotonicity of the objective function of \eqref{prb:robust-MCP}  shown above are sufficient to guarantee that a simple greedy heuristic can always yield a $(1-1/e)$ approximation. Such a greedy can simply start with a null set and iteratively add locations, one at a time, until the capacity $|S| = C$ is reached. The second and last phases of the GGX algorithm  apply due to that the fact that the objective function  $f^{\WC}(\bx)$ is differentiable in $\bx$. Such derivatives can be computed as, for any $j\in [m]$, 
\begin{align}
\frac{ \partial f^{\WC}(\bx)}{\partial x_j} &= \sum_{i\in I} q_i \frac{\partial \phi^i(\bx)/\partial 
x_j}{(1+\phi^i(\bx))^2} \nonumber \\
 &\stackrel{(a)}{=}   \sum_{i\in I} q_i \frac{ Y_j(\bv^{i*})\partial  G^i_j (\bY(\bv^{i*}) \circ \bx) }{(1+G^i (\bY(\bv^{i*}) \circ \bx))^2}, \label{eq:subprob-eq1}
\end{align}
where $\phi^i(\bx)$ is defined in \eqref{eq:define-phi} and \textit{(a)} is due to Lemma \ref{lm:lm1}. We can  apply the second phase of GGX (i.e., gradient-based local search) to further improve the solution candidate from the  greedy heuristic. At each iteration of this phase, we need to solve the following subproblem
\begin{align}
    \underset{\bx \in \{0,1\}^m}{\max} \qquad &  \nabla f^{\WC}(\overline{\bx})^\transpose \bx 
    & \label{prob:subP}\tag{\sf Sub-Prob}\\
    \text{subject to} \qquad & \sum_{j}x_j = C & \nonumber\\
    & \sum_{j\in [m],\overline{x}_{j}=1}(1-x_{j}) + \sum_{j\in [m],\overline{x}_{j}=0}x_{j} \leq \Delta, &\label{eq:subprob-eq1}
\end{align}
where $\overline{\bx}$ is the current solution candidate, $\Delta>0$ is a positive integer scalar used to define a local area around $\overline{\bx}$ in which we want to find  the next solution candidate, and \eqref{eq:subprob-eq1} is referred to as  a local branching constraint typically used  to exploit the neighborhood of a given binary solution \citep{fischetti2003local}.
The use of  $\Delta$ in the subproblem is motivated by the trust-region method in the continuous optimization literature \citep{Conn2000trust}. 
\citep{Dam2021submodularity} show that their subproblem can be solved efficiently in $\cO(m\Delta)$ but their algorithm requires that all the  coefficients of the objective function of the subproblem are non-negative. The following proposition tells us that it is  the case under our robust model,  making it possible to apply Algorithm 1 in \cite{Dam2021submodularity} to solve \eqref{prob:subP}. 
\begin{proposition}
Given any $\overline{\bx}\in \{0,1\}^m$, all the coefficients of the objective function of \eqref{prob:subP} are non-negative. As a result,  \eqref{prob:subP} can be  solved  in $\cO(m\Delta)$.
\end{proposition}
\begin{proof}
We simply use Lemma \ref{lm:lm1} and Property \textit{(iv)} of Remark \ref{prp:CPGF} to see that
\[
\frac{ \partial f^{\WC}(\overline{\bx})}{\partial x_j} =  \sum_{i\in I} q_i \frac{ Y_j(\bv^{i*}(\overline{\bx}))\partial  G^i_j (\bY(\bv^{i*}(\overline{\bx})) \circ \overline{\bx}) }{(1+G^i (\bY(\bv^{i*}(\overline{\bx})) \circ \overline{\bx}))^2} \geq 0,
\]
where $\bv^{i*}(\overline{\bx}) = \text{argmax}_{\bv^i} G^i(\bY(\bv^i)\circ \overline{\bx})$. With the positiveness of the coefficients,  \cite{Dam2021submodularity} show that \eqref{prob:subP} can be effciently solved  in $\cO(m\Delta)$.
\end{proof}

The third phase of GGX is based on steps of adding/removing locations, thus it can be applied straightforwardly to the robust problem. We briefly describe our adapted GGX in Algorithm \ref{algo:local-search} below and we refer reader to \cite{Dam2021submodularity} for more details. Here we note that, to compute $f^{\WC}(\bx)$, we need to solve the inner adversary problems $\min_{\bv^i} G^i(\bY(\bv^i) \circ \bx)$, which are convex optimization ones (Proposition \ref{prp:convexity-G}) and can be solved efficiently using a convex optimization solver. Moreover, let $\bv^{i*} = \text{argmax}_{\bv^i \in\cV_i}  G^i(\bY(\bv^i) \circ \bx)$, then using Lemma \ref{lm:lm1} we can compute the first derivatives of $f^{\WC}(\bx)$ used in \eqref{prob:subP} as
in \eqref{eq:subprob-eq1}.

\begin{algorithm}[htb]
    \caption{Local Search} 
    \label{algo:local-search}
    \SetKwRepeat{Do}{do}{until}
    \comments{1: \textbf{G}reedy heuristics}\\
\textit{   - Start from $S = \emptyset$\\
   -  Iteratively add locations to $S$, one at a time, taking  locations that increase $f^\WC (S)$ the most \\
    -  Stop when $|S| =C$ \\}
    \comments{2: \textbf{G}radient-based local search}\\
\textit{- Iteratively solve \eqref{prob:subP} to find new solution candidates\\
- Move to a new candidate solution if it yields a better objective value $f^{\WC}(S)$\\
- If \eqref{prob:subP} yields a worse solution, then reduce the size of the searching area $\Delta$\\}
    \comments{3: \textbf{Ex}changing phase}\\
\textit{- Swap one (or two) locations in $S$ with one (or two) locations in $[m]\backslash S$, taking locations that increase the objective function the most \\
- Stop when no further improvements can be made.}
 \end{algorithm}

\mtien{So far we assume that there is only a cardinality constraint $|S|\leq C$. In a real-life application, one may require  some side constraints on $\bx$ to capture, for instance,  traveling costs between facilities.  In this context, 
 for Algorithm \ref{algo:local-search} to work, one can modify Steps \#1 and \#2 to account for the additional constraints. For Step \#3, we only need to add the side constraints to the subproblem \eqref{prob:subP}. Such additional constraints, however, would break the performance guarantee secured by the greedy heuristic.}

\subsection{Outer-Approximation for the Robust MCP under MNL}
Under the MNL  model, the concavity  shown in Proposition \ref{prp:concave-MNL} suggests that  one can use a multicut outer-approximation algorithm to exactly solve the MCP. The idea is to divide the set of zones $I$ into $\cL$ disjoint groups $\cD_1,\ldots,\cD_{\cL}$ such that  $\bigcup_{l\in [\cL]} \cD_l = I$. We then write the objective function of \eqref{prb:robust-MCP} as 
\[
f^\WC(\bx) = \sum_{l\in [\cL]} \delta_l(\bx),
\]
where
\[
\delta_l(\bx) =  \sum_{i\in \cD_l} q_i -  \sum_{i\in \cD_l}\frac{q_i}{1+  \min_{\bv^i \in\cV_i} \left\{\sum_{j\in [m]} e^{v_{ij}}x_j \right\}},\:\forall l\in[\cL].
\]
A multicut outer-approximation algorithm executes by iteratively adding sub-gradient cuts to a master problem and solve it until reaching an optimal solution. Such a master problem can be defined as  
\begin{align}
    \underset{\bx\in \{0,1\}^m, \btheta \in \bbR^{L}}{\max} \qquad & \sum_{l\in [L]} \theta_l  & \label{prb:sub-OA}\tag{\sf sub-OA}\\
    \text{subject to} \qquad & \sum_{j\in[m]}x_j = C &\nonumber\\
    & \bA \bx -\bB\btheta \geq \bc &\label{eq:sub-OA-ctr1}\\
    &  \btheta \geq 0, &\nonumber
\end{align}
where the linear constraints \eqref{eq:sub-OA-ctr1} are linear cuts of the form 
$\theta_l \leq \nabla \delta_l(\bx) (\bx- \overline{\bx}) +\delta_l(\overline{\bx})$ added to the master problem \eqref{prb:sub-OA} at each iteration of the outer-approximation algorithm with a solution candidate $\overline{\bx}$. The algorithm stops when we find a solution $(\bx^*,\btheta^*)$ such that $\sum_{l\in [L]}\theta_l^* \leq \sum_{l\in [L]} \delta_l(\bx^*)$. 
Here we note that the outer-approximation algorithm 
works with any side linear constraints on $\bx$ and 
always returns an optimal solution after a finite number of iterations due to the concavity of $\delta_l(\bx)$ { and the fact that the feasible set of $\bx$ is finite.}
Similar to \eqref{eq:subprob-eq1}, the gradients of $\delta_l(\overline{\bx})$ can be computed as
\[
\frac{\delta_l(\overline{\bx})}{\partial x_j} = \sum_{i\in \cD_l} \frac{q_i \partial G^i_j(\bY(\bv^{i*}) \circ \overline{\bx})}{\left(1+ G^i_j(\bY(\bv^{i*}) \circ \overline{\bx}) \right)^2},\: \forall j\in [m].
\]
We briefly describe our multicut outer-approximation approach in Algorithm \ref{algo:MOA} below.
\begin{algorithm}[htb]
    \caption{Multicut Outer-Approximation} 
    \label{algo:MOA}
    \SetKwRepeat{Do}{do}{until}
\textit{    Initialize the master problem \eqref{prb:sub-OA}.\\ Choose a small threshold $\tau>0$ as a stopping criteria.\\
    \Do{$\sum_{l\in[L]} \overline{\theta}_l \leq \sum_{l\in [L]}\delta_l(\overline{\bx})+\tau$
    }
    {
    Solve \eqref{prb:sub-OA} to get a solution $(\overline{\bx},\overline{\btheta})$\\
    \If{$\sum_{l\in[L]} \overline{\theta}_l > \sum_{l\in [L]}\delta_l(\overline{\bx})+\tau$}
    {
     Add constraints $\overline{\theta}_l \leq \nabla \overline{\delta}_l(\overline{\bx}) (\bx- \overline{\bx}) +\delta_l(\overline{\bx})$ to \eqref{prb:sub-OA}.
    }}
    Return $\overline{\bx}$.}
\end{algorithm}

It is  possible to use the Branch-and-Cut algorithm proposed in \cite{Ljubic2018outer} to solve the robust MCP under MNL. Our multicut outer-approximation described here is similar to the one
used in \cite{MaiLodi2020_OA} and differs from the  Branch-and-Cut method of \cite{Ljubic2018outer} by the fact that it generates
\mtien{one cut per a group of \mtien{multiple} demand points instead of one cut per every single demand point, and it is a Cutting Plane approach instead of a Branch-and-Cut.}  \cite{MaiLodi2020_OA} show that the Cutting Plane approach is more efficient than the Branch-and-Cut in handling large-scale instances.

\section{Numerical Experiments}
\label{sec:expertiments}
We provide numerical experiments showing the performance of our robust approach in protecting us from worst-case scenarios when the choice parameters are not known with certainty. We first discuss our approach to construct uncertainty sets to capture the issue of choice parameter uncertainties. We then provide experiments with MCP instances under two popular discrete choice models in the  GEV family, namely, the MNL and nested logit models.

\subsection{Constructing Uncertainty Sets}
 We first discuss our approach to build uncertainty sets to capture uncertainty issues when identifying choice parameters in the MCP. In the deterministic setting, it is assumed that there is only one vector of customer choice utilities $\bv^i$ for each zone $i\in I$. However, the real market typically has many different types of customers characterized by, for instance, age or income. The choice parameters may significantly vary across different customer types. Let us assume that, for each zone $i\in I$,  there are $N$ types of customers  with $N$ utility vectors {$\widetilde{\bv}^{i1},...,\widetilde{\bv}^{iN}$}, respectively. We  let $\tau_{i1},...,\tau_{iN}\in [0,1]$ be the actual proportions of the  customer types with $\sum_{n\in [N]}\tau_{in}=1$. If these proportions are known with certainty, then one can define a  mean value vector $\widetilde{\bv}^i =\sum_{n\in [N]} \tau^{in}\widetilde{\bv}^{in}$ and  solve the corresponding deterministic MCP problem. Our assumption here is that these proportions cannot be identified exactly,  but we know that the actual proportions are not too far from some estimated proportions. In this context,  uncertainty sets can be defined as  
 \begin{equation}
\label{eqn:uncertainty_set}
    \cV_i = \left \{\bv^i=\sum_{n\in [N]}\eta_{n}\widetilde{\bv}^{in} \Big| \eta_n\geq 0,\forall n\in [M]; \sum_{n\in [N]}\eta_{n}=1;\quad \text{and}\quad || \pmb{\eta} -\widetilde{\btau}^i|| \leq \epsilon\right\},
 \end{equation}
 where $\widetilde{\btau}^i = \{\widetilde{\tau}^{in},\;n\in [N]\}$ are some estimates of the actual  proportions $\btau^i$ and   $\epsilon\in [0,1]$ represents the ``uncertainty level'' of the uncertainty set. Larger $\epsilon$ values lead to larger uncertainty sets, thus resulting in  more conservative models that may help provide better protection against  worst-case scenarios, but may give a low average performance. In contrast, smaller $\epsilon$ values provide smaller uncertainty sets, which would lead to better average performance but may give a bad performance in protecting against worst-case scenarios. The firm could adjust $\epsilon$ to balance the worst-case protection and average performance. Clearly, if $\epsilon=0$, the uncertainty sets become singleton and the robust MCP becomes a deterministic MCP with mean-value choice utilities, i.e., the actual proportions are known perfectly, or the firm just treats their estimated proportions as the actual ones. On the other hand, if we select  $\epsilon>N$, then
 the uncertainty sets will cover all possible affine combinations of $\{\widetilde{\bv}^i,\;i\in I\}$. This reflects the situation that the firm knows nothing about the actual proportions and  has to ignore the pre-computed values  $\widetilde{\btau}^i$. This way of constructing uncertainty sets is similar to the approach employed in \cite{rusmevichientong2012robust} in the context of assortment optimization under uncertainty.

\subsection{Baseline Approaches and Other Settings}
We will compare our robust approach (denoted as RO) against other baseline approaches under
 the MNL and nested logit models.
The first baseline approach, which is a deterministic one (denoted as DET1), relies on mean-value choice parameters, i.e., we solve the MCP with average values of the choice utilities, i.e.  $\widetilde{\bv}^i = \sum_{n\in [N]}\tau_{in} \widetilde\bv_{in}$.
This is  the case of $\epsilon= 0$, i.e., no uncertainty in the robust model. 
For the second baseline approach (denoted as DET2), we utilize all the vectors of choice parameters
$\widetilde\bv_{in}$, for all $n\in [N]$ and solve the following mixed version of the MCP
\begin{equation}
\label{prob:MixedMNL}
 \max_{S \in \cK}\left\{ \sum_{i\in I}q_i \sum_{n\in [N]} \tau^{in}\left(\sum_{j\in S} P(j|\widetilde\bv_{in}, S)\right) \right\}.    
\end{equation}
So, for the DET1 and DET2 approaches, we acknowledge that there are $N$ customer types in the market but ignore the uncertainty issue. 
Note that the DET2 approach can be viewed as an MCP under a mixed logit model.
We  consider another baseline that accounts for the issue that the proportion of the customer types maybe not be determined with certainty. We perform this approach by sampling over the uncertainty sets $\cV_i$. Then, for each sample of the utilities, we solve the corresponding MCP to get a solution. Since we aim at covering the worst-case scenarios, for each solution, we  sample again from the uncertainty sets to get an approximation of the worst expected captured demand value. We then pick a solution that gives the best worst-case objective value over samples. This way allows us to find solutions that are capable of guaranteeing some protection against worst-case scenarios. \mtien{This approach can be viewed as a sampling-based method to solve the robust problem.
We denote it as SA.}      

For all the MNL and nested logit instances,  we employ the local search (Algorithm \ref{algo:local-search})
 to solve the robust/deterministic MCP, for all the RO, DET1, DET2, and SA approaches, noting that for MNL instances, the outer-approximation (Algorithm \ref{algo:MOA}) can be used as an exact method, but it is generally outperformed by the local search in terms of both solution quality and running time. \mtien{For nested logit instances, the local search is also more convenient to use, as (i) the multicut outer-approximation algorithm  becomes heuristic and does not offer us any guarantees, and (ii)} \cite{Dam2021submodularity} already show that the local search outperforms the multicut outer-approximation algorithm for nested instances. Thanks to the submodularity, the local search  always gives us at least $(1-1/e)$ approximation solutions (Corollary \ref{coro:co1}).

We use instances from the three datasets HM14, ORlib, and NYC to generate instances for our robust problems and we refer the reader to \cite{FLO_Freire2016branch} for a detailed description. These datasets have  been used in previous MCP studies \citep{Ljubic2018outer,MaiLodi2020_OA,Dam2021submodularity}.  
We select
$N = 5$ (i.e., there are 5 customer types in each zone). We then randomly choose the pre-determined proportions $\{\tau_{1},...,\tau_{5}\}$ such that $\sum_{n\in[N]}\tau_n =1$. For each deterministic instance from the datasets (HM14, ORlib, or NYC) {to generate the underlying utility vectors  $\{\widetilde{\bv}^{1i},...,\widetilde{\bv}^{5i}\}$ to construct uncertainty sets, 
we take the set of  utility values $\{v^i_j,\;i\in I, j\in [m]\}$ from the data and sample (randomly and uniformly)  each  element $\widetilde{v}^{ki}_{j}$ in range {$[0.7\times v^{i}_j;1.3\times v^{i}_j]$}, for all $k =1,\ldots,5$.} 
To compare the performances of the RO, DET, and SA approach in protecting us from worst-case scenarios, we vary the uncertainty level $\epsilon$  and compare the robust solutions with those from the DET and SA approaches. More precisely, we will perform the following steps  for each $\epsilon > 0$.
\begin{itemize}
    \item[(i)] Define an uncertainty set as in \eqref{eqn:uncertainty_set}.
    \item[(ii)] For RO, we solve the robust problem \eqref{prb:robust-MCP} and obtain a robust solution $\bx^{\RO}$.
    \item[(iii)]  For DET1, we solve the deterministic MCP \eqref{prob:MCP-1} with the weighted average utilities $\widetilde{\bv} = \sum_{n\in [n]}\tau_{n}\widetilde{v}_{n}$ and obtain a solution $\bx^{\DET}$.
    \item[(iv)] For DET2,  we solve \eqref{prob:MixedMNL} to get a solution $\bx^{\MDET}$.
    \item[(v)] For SA, we take 10 samples from the uncertainty sets to obtain 10 candidate solutions.\footnote{{More samples can be taken, but we restrict ourselves to 10 samples, as the SA approach is  expensive to run as we will show later.}} Then, for each solution obtained,  we use 1000 samples of the utilities,  taken from the uncertainty sets, to get an approximate worst-case objective value.  We then pick the  solution $\bx^{\SA}$ that gives the best  worst-case objective. All the sampling is done randomly and uniformly. 
    \item[(vi)]\mtien{We assess the value  and the  price  of the robust approach by checking the performances of different solutions obtained by the robust, sampling-based, and deterministic approaches in the uncertain environment, and in the nominal (or deterministic) setting. That is, we plug $\bx^{\DET}$, $\bx^{\MDET}$, and $\bx^{\SA}$ into the robust problem to compute relative decreases in terms of worst-case objective, and plug  $\bx^{\RO}$  and $\bx^{\SA}$ into the deterministic MCP to evaluate relative decreases in terms of average objective. These two measurements are often referred to as the value and the price of robustness and have been popularly used to assess robust optimization methods \citep{mehmanchi2020robust,bertsimas2004price} 
    }
 \item[(vii)] \mtien{We additionally evaluate the performances of $\bx^{\RO}$, $\bx^{\DET}$, $\bx^{\MDET}$, and $\bx^{\SA}$ by looking at the empirical distributions of the objective values yielded by different vectors of utilities sampled from the uncertainty sets.} To this end,  we randomly and uniformly sample $2000$ utility vectors $\{\bv^i,\;i\in I\}$  from the uncertainty sets  and compute, for each utility sample, the corresponding objective values  given by $\bx^{\RO}$, $\bx^{\DET}$, $\bx^{\MDET}$, and $\bx^{\SA}$.
This allows us to plot and analyze  \mtien{empirical} distributions of the objective values given by different solutions 
when the \mtien{uncertain} utilities $\bv^i$ vary randomly within the uncertainty sets.  
\end{itemize}

We use MATLAB 2020 to implement and run the algorithms. 
We  use the maximum norm  to define the uncertainty sets in \eqref{eqn:uncertainty_set}. Under this setting, the  adversary's optimization problems can be formulated as convex optimization ones with (strictly) convex objective functions and linear constraints. 
We use the nonlinear optimization solver \textit{fmincon} from MATLAB, under default settings, to solve the adversary's convex optimization problems (i.e., $\min_{\bv^i \in } G^i(\bY(\bv^i)\circ \bx)$). 
The experiments were done on a PC with processor AMD-Ryzen 7-3700X CPU-3.80 GHz and with
16 gigabytes of RAM.


\subsection{Comparison Results}
In this section, we present comparison results for MNL and nested logit instances to assess the value of the robust approach. Computing time comparisons will also be presented.

\subsubsection{MNL Instances} 

\begin{figure}[htb]
  \includegraphics[width=\linewidth]{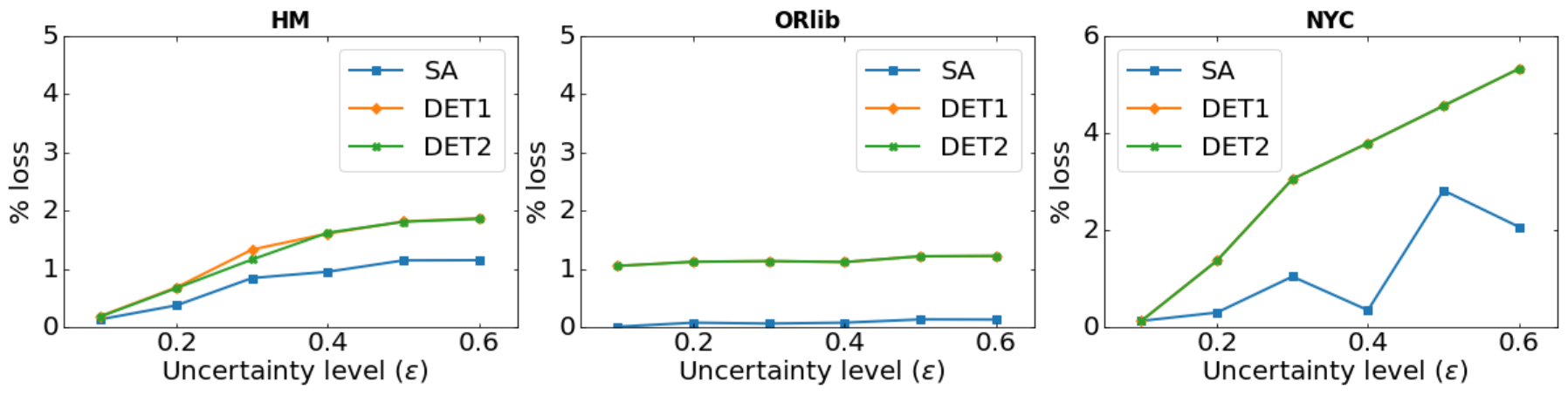} 
  \caption{\mtien{Value of robustness for MNL instances; the performances of DET1 and  DET2 are almost identical.}}
  \label{fig:MNL-VoR}
\end{figure}

\begin{figure}[htb]
  \includegraphics[width=\linewidth]{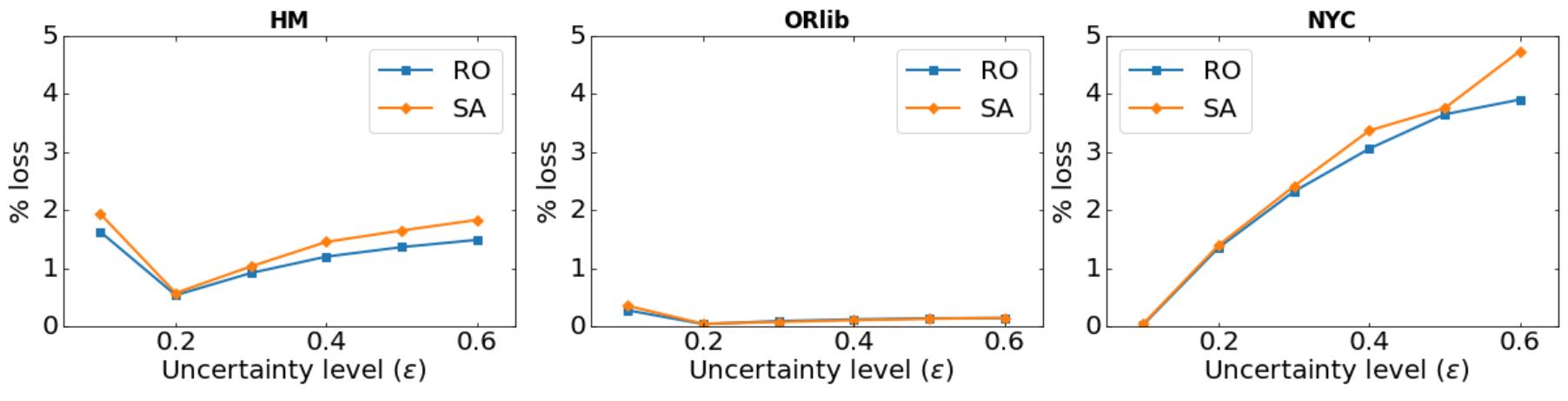} 
  \caption{\mtien{Price of robustness for MNL instances.}}
  \label{fig:MNL-PoR}
\end{figure}

\mtien{We first solve MNL instances using Algorithm \ref{algo:local-search} and plot, in Figure \ref{fig:MNL-VoR} and Figure \ref{fig:MNL-PoR}, the average relative decreases (``\% loss'') w.r.t. the robust  and deterministic MCP for the three datasets (i.e. HM, ORlib, and NYC).  In both figures, higher ``\% loss'' implies worse results. The results generally show that the percentage loss, as expected, seems to increase as the uncertainty level $\epsilon$ increases. The percentage losses for ORlib are remarkably smaller than the losses for the other datasets, and the losses for the largest dataset NYC are, as expected, significantly higher than those from ORlib and HM. Moreover,  Figure \ref{fig:MNL-VoR}  shows that SA performs better than DET1 and DET2 in terms of the value of robustness. 
 Looking at  both figures, it can be  observed that the deterministic solution performs worse in the robust setting than the robust and SA solutions do in the deterministic setting for the ORlib dataset, and performs comparably for the other datasets. }  
 
 
 \mtien{To further assess the performance of the robust approach, we pick instances of size $(|I|,m) = (100, 50)$ and plot, in Figure \ref{fig:MNL}, empirical distributions of the objective values  given by the four approaches and utilities sampled from the uncertainty sets. More plots can be found in Appendix \ref{appd:exp}. }  
 The first  row of the figure shows the histograms for small $\epsilon$ ($\epsilon\leq0.08$). Under these small \textit{uncertainty levels}, the distributions given by the four approaches are quite similar and there is no clear advantage of the robust approach. This is an expected observation, as when $\epsilon$ is small, the corresponding uncertainty sets become small and $\bx^{\RO}$ should be close to $\bx^{\DET}$, $\bx^{\MDET}$, and $\bx^{\SA}$. 
 When we increase $\epsilon$, the second row of Figure \ref{fig:MNL}  shows clear differences between the four approaches. {That is,  the distributions given by $\bx^{\RO}$ always have lower variances, and shorter tails, as compared to the those from the other approaches.}
 In terms of worst-case protection, the RO performs the best, followed by the SA and then the two deterministic approaches.  
 This clearly demonstrates the capability of the RO approach in giving ``\textit{not-too-low}'' expected captured demands. Moreover, the protection against ``low'' objective values seems higher as $\epsilon$ increases. This is also an expected observation, as the RO and SA approaches are essentially designed  for this purpose. We also observe that the two DET approaches perform better in terms of best-case scenarios, indicating a trade-off  of being conservative when making robust decisions under uncertainty. 


\begin{figure}[htb]
  \centering
  \includegraphics[width=\linewidth]{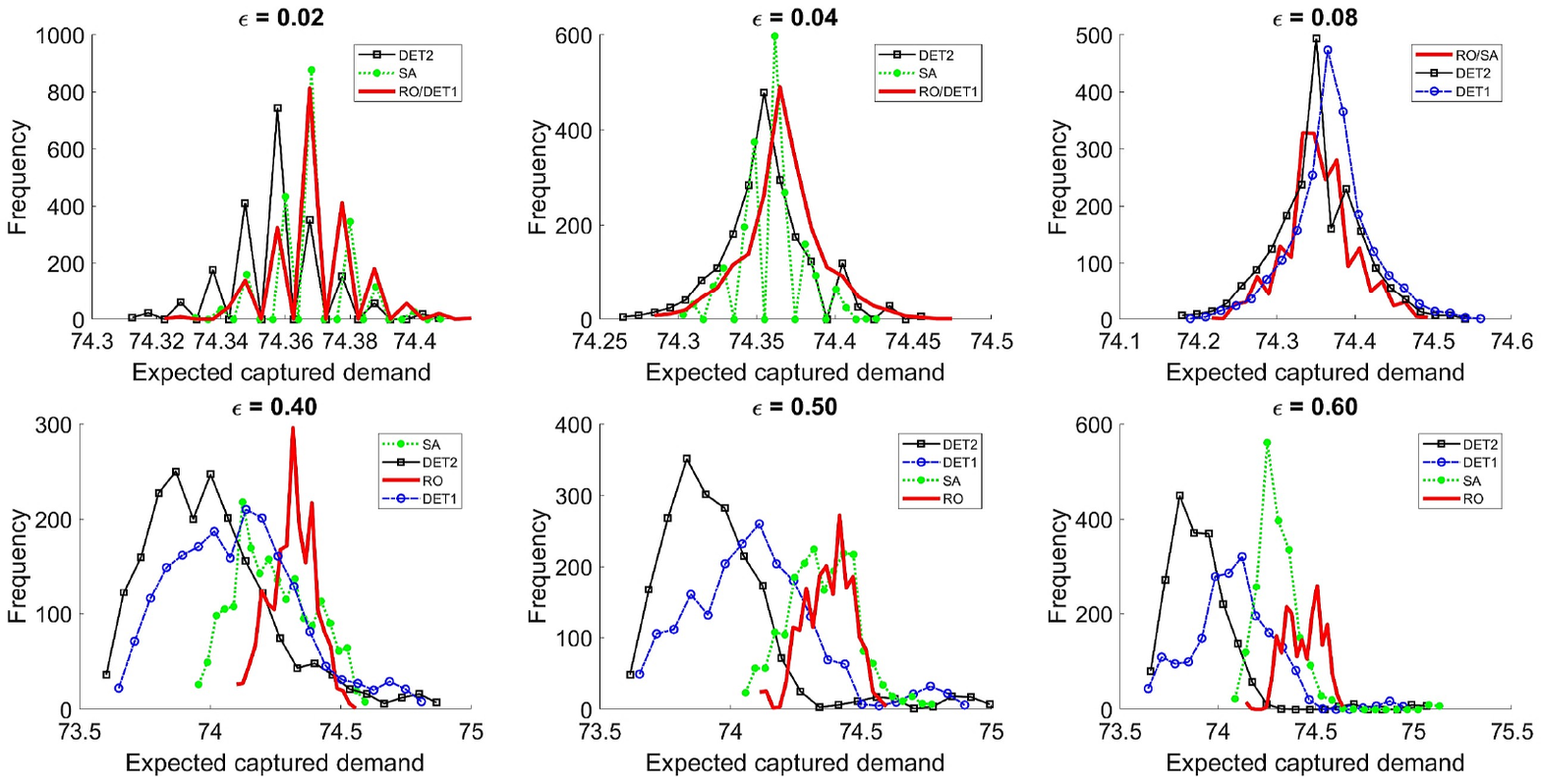}  
  \caption{Comparison of the distributions of the objective values given by solutions from RO, SA, DET1 and DET2 approaches, under the MNL choice model and  instances of size $|I|=100$ and $m = 50$.}
 \label{fig:MNL}
\end{figure}

\mtien{Before moving to nested logit instances, we note that the robust MCP under MNL is relevant to the robust fractional 0-1 program studied in \cite{mehmanchi2020robust}. We provide a comparison with this work in Appendix \ref{sec:comp-meh}.}



\subsubsection{Nested Logit Instances}

We now provide comparison results for nested logit instances. As mentioned, to create the instances, we divide the set of locations into 5 groups of the same size (i.e, $L=5$) and choose the nested logit parameters as $\mu = (1.1,1.2,1.3,1.4,1.5)$, noting that these parameters are just selected at random for a testing purpose and other values can be chosen. \mtien{Similar to the MNL instances, we first report the value and price of robustness in Figures \ref{fig:Nested-VoR} and \ref{fig:Nested-PoR} below. We first notice that the percentage losses reported in Figure \ref{fig:Nested-VoR} are remarkably smaller than those reported for the MNL instances. This would be because we  solve the nested logit instances by the local search algorithm, thus the solutions $\bx^{\RO}$ obtained may not be optimal for the robust problem, affecting the value of robustness. 
Moreover, in analogy to MNL instances,  SA performs slightly better than DET1 and DET2, in terms of the value of robustness.
We also see that the percentage losses for ORlib are smaller, compared to losses for the other datasets. Especially, the percentage losses of the RO  and SA solutions in the deterministic environment almost vanish. The losses for NYC instances, similarly to the MNL case, are also significantly larger than the losses for ORlib and HM. Moreover, while the robust solutions perform better in the deterministic problem than the DET solutions do in the robust setting for the ORlib dataset, they perform slightly worse for the other datasets. }

\begin{figure}[htb]
  \includegraphics[width=\linewidth]{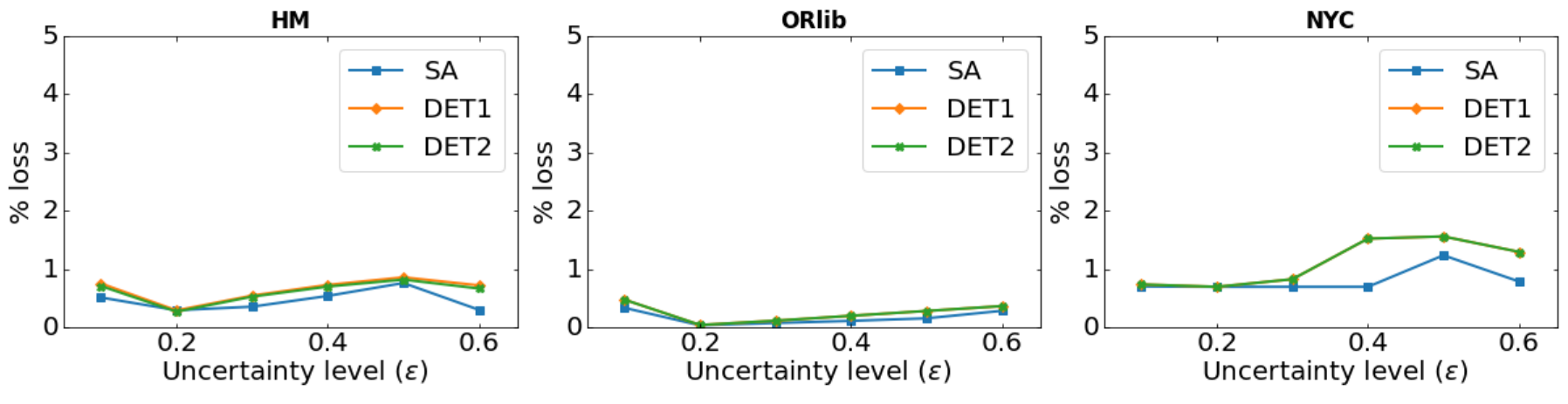} 
  \caption{\mtien{Value of robustness for nested logit instances; the performances of DET1 and DET2  are almost identical.}}
  \label{fig:Nested-VoR}
\end{figure}

\begin{figure}[htb]
  \includegraphics[width=\linewidth]{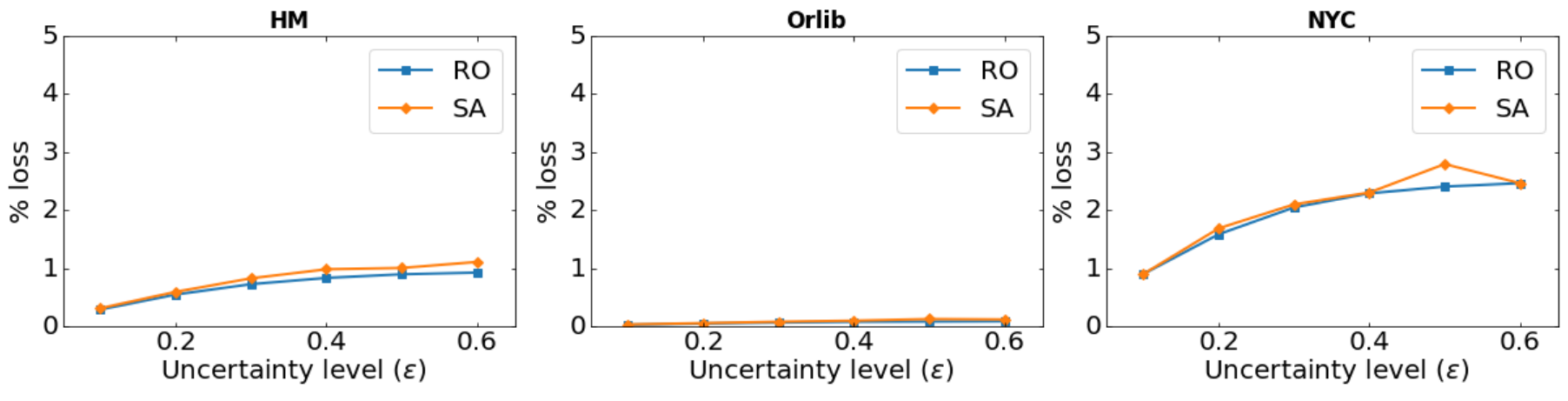} 
  \caption{\mtien{Price of robustness for nested logit instances.}}
  \label{fig:Nested-PoR}
\end{figure}

We also pick instances of size $(|I|,m) = (100, 50)$ to plot empirical distributions of the objective values  given by the four approaches and utilities sampled from the uncertainty sets.  
 In the first row of Figure \ref{fig:Nested},
we plot the histograms for $\epsilon\in \{0.02,0.04,0.08\}$. As we can see, histograms given by RO, DET1, and SA are identical. For $\epsilon = 0.02$ or $\epsilon = 0.04$, the histograms look very similar, but for $\epsilon = 0.08$ we start seeing that the histogram given by RO has { smaller variance and shorter tail}. Moreover, even though the differences are not clear with these small uncertainty set levels, some protection from the RO approach against bad scenarios can still be observed.

\begin{figure}[htb]
  \includegraphics[width=\linewidth]{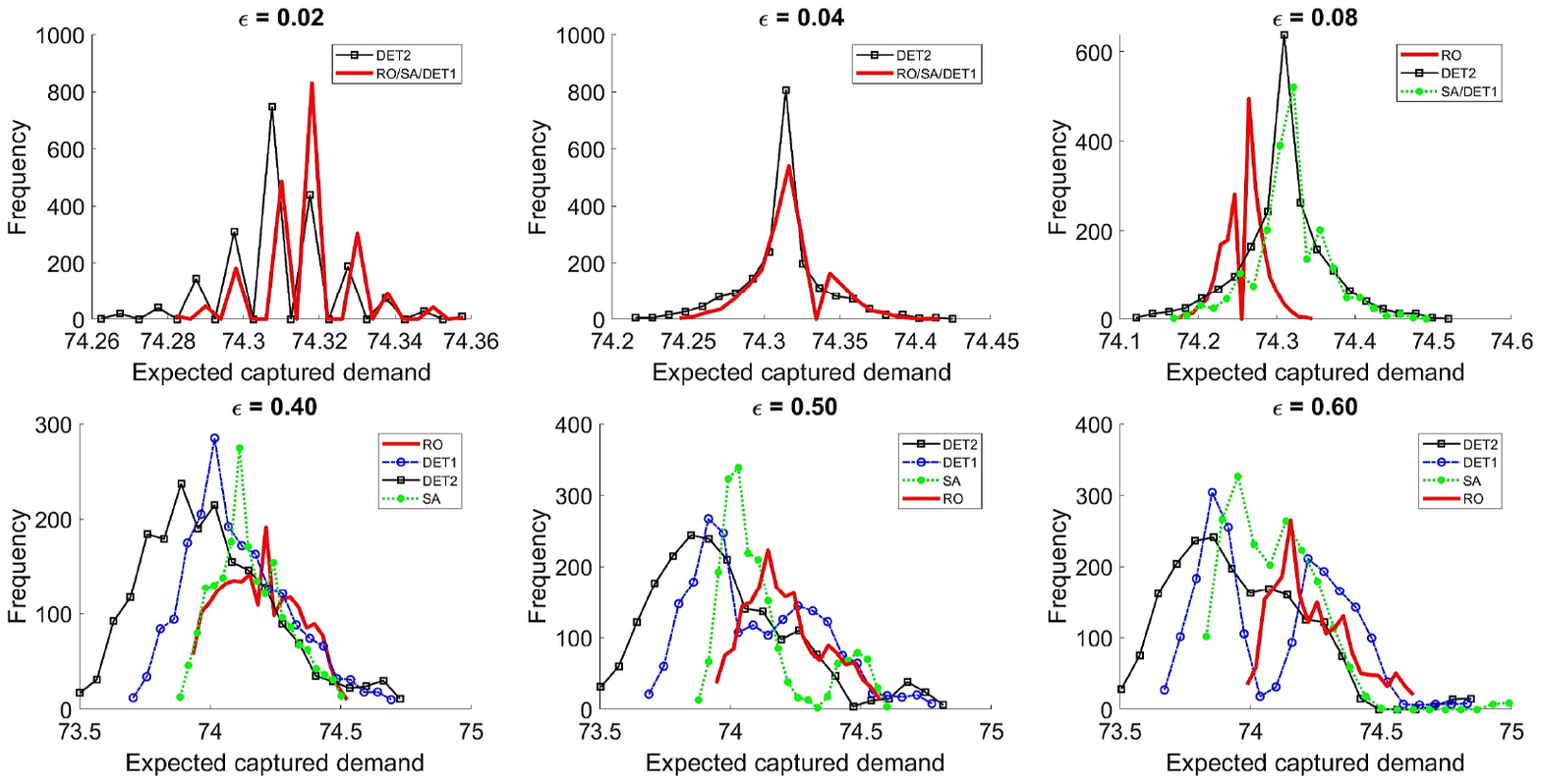}  
  \caption{Comparison of the distributions of the objective values given by solutions from RO, SA, DET1 and DET2 approaches, under the nested logit choice model and instances of size $|I|=100$ and $m = 50$.}
  \label{fig:Nested}
\end{figure}

Histograms with larger $\epsilon$ ($\epsilon \in {\{0.4,0.5,0.6\}}$) are plotted in the second row of Figure \ref{fig:Nested}. There is no surprise, as similar to the MNL instances, the distributions given by the RO approach have small variances, shorter tails, and larger worst-case objective values, as compared to those given by the other approaches. In particular, we can see that the two deterministic approaches can give many low objective values. The SA approach seems to do better in protecting  the objective  value from being too low, and RO performs the best in pushing its worst-case scenarios to higher objective  values. It is worth noting that the DET1 approach (when average values of the choice utilities are made use of) performs better than the mixed deterministic approach DET2 in terms of worst-case protection. A trade-off between having high worst-case objective values and having low best-case objective values can also be  observed. This is also consistent with remarks from prior studies in the context of assortment and pricing optimization where discrete choice models are also employed \citep{rusmevichientong2012robust, LI2019, mai2019robust}. 




\subsubsection{Percentile Ranks of the RO’s Worst-case Objective Values}
We look closely into the distributions of the objective values given by the four approaches to see where the RO's worst-case values are located in other distributions. This would help evaluate how much RO can provide protection when $\epsilon$ increases. To this end, we compute the percentile ranks of the RO's worst-case objective values in the distributions given by the DET and SA solutions. Such a percentile rank can tell us the percentage of the values in the DET1, DET2, or SA distributions that are lower than or equal to the worst objective values given by the RO approach. We plot the percentile ranks of the RO's worst objective values in Figure \ref{fig:PerRank_mnl_1} for both the MNL and nested logit models.
For the MNL instances,   when $\epsilon>0.1$,  the percentile ranks are significant and increasing from about 0\% to almost 100\% for DET2, from 0\% to about 60\%  for DET1, and are quite small for SA (less than 20\%). For the nested logit instances, the percentile ranks for DET1 and DET2 can go up to more than 40\% and can go up to 30\%  for the SA approach. This clearly indicates how much RO can provide protection, as compared to the other approaches, for instance, for the nested model with $\epsilon = 0.6$, the RO's worst objective values are larger than more than 30\% of the SA's sample objective values, and more than 40\% of the DET1's and DET2's sample objective values. It can  be seen that, among the DET1, DET2 and SA approaches, SA performs better than DET1 and DET2, and DET1 is better than DET2, in terms of worst-case protection. 

\begin{figure}[htb]
  \centering
  \includegraphics[width=0.8\linewidth]{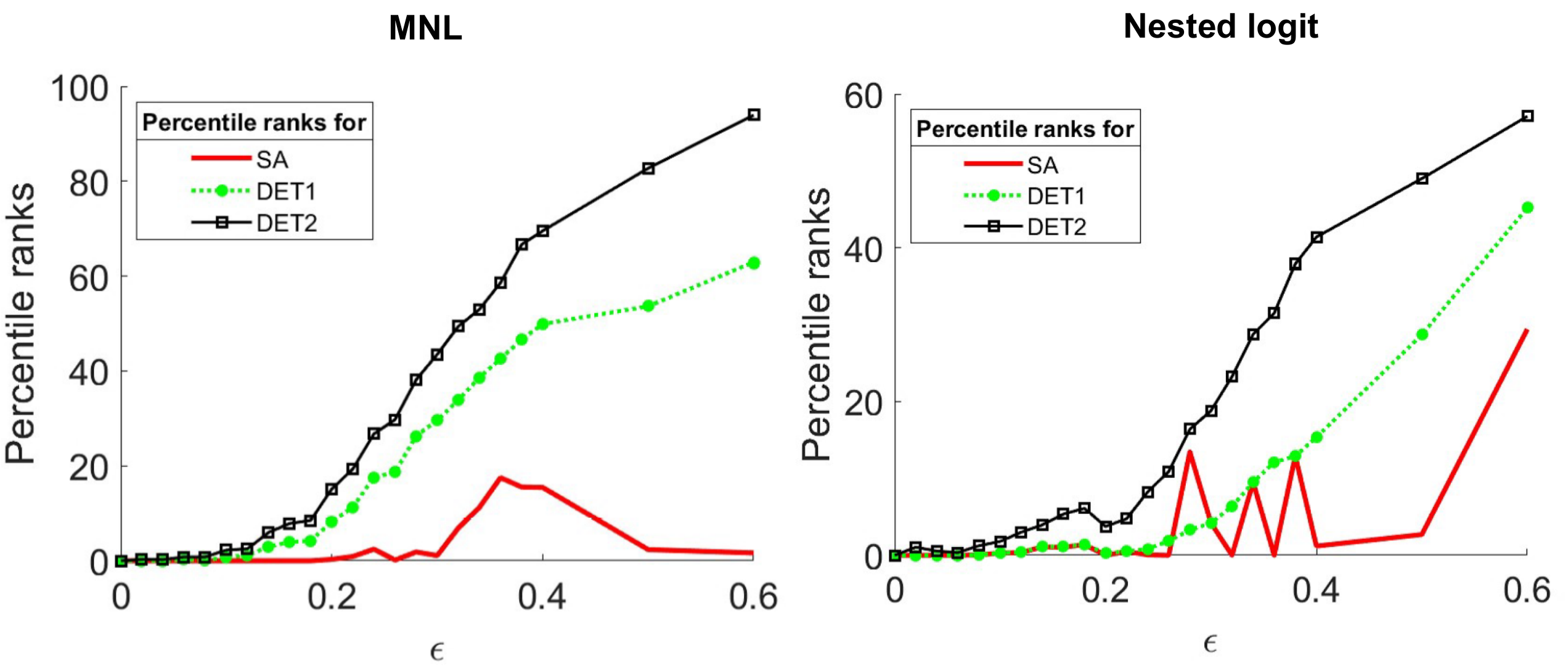}  
  \caption{The percentile ranks of the RO's worst-case values in the distributions given by  the SA, DET1 and DET2 solutions under the MNL and nested choice models for instance of $|I|=100$ and $m=50$.}
 \label{fig:PerRank_mnl_1}
\end{figure}

\subsection{Computing Time Comparison}
We   look at the computing times required by the different approaches to terminate. In Table \ref{tab:computing}, we report the average computing times for the four approaches, under the two discrete choice models (the MNL and nested logit) and with {two} instance sizes. It clearly shows that DET1 require the least amount of running time. It is interesting to see that RO requires less computing times than DET2, except for the instances of size $|I| = 82341$.
 SA is much more time-consuming, as compared to the other three approaches since it requires sampling and solving several deterministic problems. On the other hand, for the largest problem,  RO requires the largest computing times under the MNL model due to the fact that it needs to solve a large number of inner minimization problems (i.e., solving $\min_{\bv^i \in \cV_i } G^i(\bY(\bv^i)\circ \bx)$ for  $i=1,2,...,82341$).

\begin{table}[htb]
\begin{center}
\begin{tabular}{l|l|l|l|l|l|l}
Choice model & $|I|$ & $m$ & RO & SA & DET1 & DET2 \\ \hline
\multicolumn{1}{c|}{\multirow{3}{*}{MNL}} & 100 & 50 & 14.0 & 18.8 & {0.3} & {3.6} \\
\multicolumn{1}{c|}{} & 100 & 100 & 12.4 & 21.6 & {0.3} & {1.5} \\
\multicolumn{1}{c|}{} & 200 & 100 & 26.7 & 25.09 & {0.3} & {13.7} \\ 
\multicolumn{1}{c|}{} & {1000} & 100 & 60.0 & 479.72 & 45.3 & 680.4 \\
\multicolumn{1}{c|}{} & {82341} & 59 & 2861.1 & 885.1 & 0.39 & 6.5 \\\hline
\multirow{3}{*}{Nested logit} & 100 & 50 & 5.6 & 17.0 & {1.4} & {3.3} \\
 & 100 & 100 & 11.3 & 88.8 & {8.2} & {6.6} \\
 & 200 & 100 & 16.40 & 116.4 & {11.0} & {23.3}\\
 & {1000} & 100 & 49.0 & 311.6 & 23.2 &  155.7  \\
 & {82341} & 59 & 2262.4 & 4386.9 & 311.5 & 2586.7
\end{tabular}%
\end{center}
\caption{{Average computing times (in seconds).}}
\label{tab:computing}
\end{table}

As shown in \cite{Dam2021submodularity}, such a local search procedure as in Algorithm \ref{algo:local-search} achieves the best performance in terms of  objective value and computing time, as compared to outer-approximation algorithms or MILP reformulations. Thus,  we  provide here a comparison of  Algorithm \ref{algo:local-search} and Algorithm \ref{algo:MOA}. Note that, for the MCP under  the nested logit model, the objective function is highly non-concave  and the outer-approximation method often performs much worse than the local search heuristic \citep{Dam2021submodularity}. Thus, we only provide a comparison of the two algorithms with MNL instances. 
We use Algorithm \ref{algo:local-search} and Algorithm \ref{algo:MOA} to solve the robust MCP with MNL instances of different sizes with $\epsilon \in \{0.02, 0.04, 0.08, 0.4, 0.5, 0.6\}$
and report the objective values and the computing times in  Table \ref{tab:compare_MOA_LS}. For ease of exposition, we denote Algorithm \ref{algo:local-search} as LS and Algorithm \ref{algo:MOA} as MOA. We observe that the LS and MOA always return the same objective values for all the instances, implying that LS always give us optimal solutions. 
In addition, the computing times required by MOA and LS are  similar for the  instances of sizes $(|I|, m) = (100, 50)$, $(|I|, m) = (100, 100)$ and $(|I|, m) = (200, 100)$. However, for  instances of size $(|I|, m) = (1000, 100)$, LS requires shorter computing times than MOA. In particular, the computing times of LS are always below 40 seconds when MOA always needs more than 60 seconds to finish. For the largest instances $(|I|, m) = (82341, 59)$, LS requires shorter computing times for small $\epsilon$ value (i.e, $\epsilon \leq 0.08$). In contrast, MOA requires shorter computing times for the largest instances with large $\epsilon$ values (i.e., $\epsilon\geq 0.4$). 

\begin{table}[htb]
\centering
\begin{tabular}{|l|l|l|ll|ll|}
\multicolumn{1}{|c|}{\multirow{2}{*}{$|I|$}} & \multicolumn{1}{c|}{\multirow{2}{*}{$m$}} & \multicolumn{1}{c|}{\multirow{2}{*}{$\epsilon$}} & \multicolumn{2}{c|}{MOA} & \multicolumn{2}{c|}{LS} \\ \cline{4-7} 
\multicolumn{1}{|c|}{} & \multicolumn{1}{c|}{} & \multicolumn{1}{c|}{} & \multicolumn{1}{l|}{Objective values} & \multicolumn{1}{l|}{Time} & \multicolumn{1}{l|}{Objective values} & \multicolumn{1}{l|}{Time} \\ \hline
100 & 50 & 0.02 & \multicolumn{1}{l|}{74.33} & \textbf{2.24} & \multicolumn{1}{l|}{74.33} & 2.79 \\
100 & 50 & 0.04 & \multicolumn{1}{l|}{74.29} & \textbf{2.38} & \multicolumn{1}{l|}{74.29} & 2.71 \\
100 & 50 & 0.08 & \multicolumn{1}{l|}{74.22} & \textbf{2.68} & \multicolumn{1}{l|}{74.22} & 2.72 \\
100 & 50 & 0.4 & \multicolumn{1}{l|}{74.01} & \textbf{2.21} & \multicolumn{1}{l|}{74.01} & 2.27 \\
100 & 50 & 0.5 & \multicolumn{1}{l|}{74.01} & \textbf{1.83} & \multicolumn{1}{l|}{74.01} & 1.93 \\
100 & 50 & 0.6 & \multicolumn{1}{l|}{74.02} & \textbf{1.76} & \multicolumn{1}{l|}{74.02} & 1.82 \\
100 & 100 & 0.02 & \multicolumn{1}{l|}{74.63} & \textbf{3.08} & \multicolumn{1}{l|}{74.63} & 3.22 \\
100 & 100 & 0.04 & \multicolumn{1}{l|}{74.60} & 3.46 & \multicolumn{1}{l|}{74.60} & \textbf{3.44} \\
100 & 100 & 0.08 & \multicolumn{1}{l|}{74.55} & 4.04 & \multicolumn{1}{l|}{74.55} & \textbf{3.99} \\
100 & 100 & 0.4 & \multicolumn{1}{l|}{74.40} & \textbf{2.40} & \multicolumn{1}{l|}{74.40} & 2.47 \\
100 & 100 & 0.5 & \multicolumn{1}{l|}{74.40} & \textbf{2.07} & \multicolumn{1}{l|}{74.40} & 2.24 \\
100 & 100 & 0.6 & \multicolumn{1}{l|}{74.39} & \textbf{2.03} & \multicolumn{1}{l|}{74.39} & 2.08 \\
200 & 100 & 0.02 & \multicolumn{1}{l|}{148.71} & \textbf{5.92} & \multicolumn{1}{l|}{148.71} & 6.29 \\
200 & 100 & 0.04 & \multicolumn{1}{l|}{148.64} & \textbf{6.53} & \multicolumn{1}{l|}{148.64} & 6.89 \\
200 & 100 & 0.08 & \multicolumn{1}{l|}{148.51} & \textbf{7.28} & \multicolumn{1}{l|}{148.51} & 7.68 \\
200 & 100 & 0.4 & \multicolumn{1}{l|}{148.15} & \textbf{4.62} & \multicolumn{1}{l|}{148.15} & 4.83 \\
200 & 100 & 0.5 & \multicolumn{1}{l|}{148.14} & \textbf{4.13} & \multicolumn{1}{l|}{148.14} & 4.23 \\
200 & 100 & 0.6 & \multicolumn{1}{l|}{148.14} & \textbf{3.84} & \multicolumn{1}{l|}{148.14} & 4.00 \\
1000 & 100 & 0.02 & \multicolumn{1}{l|}{39236.30} & 60.10 & \multicolumn{1}{l|}{39236.30} & \textbf{29.83} \\
1000 & 100 & 0.04 & \multicolumn{1}{l|}{39228.52} & 78.44 & \multicolumn{1}{l|}{39228.52} & \textbf{31.78} \\
1000 & 100 & 0.08 & \multicolumn{1}{l|}{39213.34} & 82.75 & \multicolumn{1}{l|}{39213.34} & \textbf{36.19} \\
1000 & 100 & 0.4 & \multicolumn{1}{l|}{39153.40} & 75.73 & \multicolumn{1}{l|}{39153.40} & \textbf{29.17} \\
1000 & 100 & 0.5 & \multicolumn{1}{l|}{39150.54} & 70.80 & \multicolumn{1}{l|}{39150.54} & \textbf{24.28} \\
1000 & 100 & 0.6 & \multicolumn{1}{l|}{39149.27} & 69.19 & \multicolumn{1}{l|}{39149.27} & \textbf{22.81} \\
82341 & 59 & 0.02 & \multicolumn{1}{l|}{71682.01} & 2915.80 & \multicolumn{1}{l|}{71682.01} & \textbf{2833.67} \\
82341 & 59 & 0.04 & \multicolumn{1}{l|}{71603.14} & 3034.34 & \multicolumn{1}{l|}{71603.14} & \textbf{2962.36} \\
82341 & 59 & 0.08 & \multicolumn{1}{l|}{71458.12} & 3181.14 & \multicolumn{1}{l|}{71458.12} & \textbf{3109.69} \\
82341 & 59 & 0.4 & \multicolumn{1}{l|}{70996.35} & \textbf{2384.12} & \multicolumn{1}{l|}{70996.35} & 2465.88 \\
82341 & 59 & 0.5 & \multicolumn{1}{l|}{70985.87} & \textbf{2303.91} & \multicolumn{1}{l|}{70985.87} & 2395.21 \\
82341 & 59 & 0.6 & \multicolumn{1}{l|}{70982.53} & \textbf{2237.14} & \multicolumn{1}{l|}{70982.53} & 2326.52
\end{tabular}
\caption{{Comparison of the MOA and LS algorithms.}}
\label{tab:compare_MOA_LS}
\end{table}

\section{Conclusion}
\label{sec:concl}
We have formulated  and studied a robust version of the MCP under GEV models. We have shown that the adversary's minimization problem can  be solved by convex optimization, and the robust model preserves the monotonicity and submodularity from its deterministic counterpart, leading to the fact that a simple greedy heuristic can guarantee $(1-1/e)$ approximation solutions. We have then introduced the  GGX algorithm that works with the robust MCP under GEV, and a multicut outer-approximation algorithm that can exactly solve the robust MCP under MNL. Our numerical experiments based on the MNL and nested logit models have shown the advantages of our model and algorithms in providing protection against worst-case scenarios, as compared to other deterministic and sampling-based baseline approaches.  

Our robust model assumes that the choice parameters can vary uniformly in the uncertainty sets, which might be conservative in some contexts where the distribution of these parameters may be partially known. Therefore, it would be interesting to consider a distributionally robust version of the MCP. 
It is  interesting to look at a  model where the firm and the competitor  make decisions in a Stackelberg  game setting.
Moreover, in this paper we assume that the (expected) number of customer in each zone is fixed and ignore any uncertainty associated with the fact that customers may travel between different zones. Accounting for this would require a new model to predict how customers would move between zones any maybe a new stochastic or robust optimization model for the MCP. This would  be  
an interesting direction for future work. 

As mentioned earlier, our model only targets uncertainties associated  with the deterministic parts of the utilities. There would be other sources of uncertainty  that would come from, for instance, the structure of the GEV model or the distribution of the random parts of the utilities. Models and algorithms addressing such uncertainties would be challenging,  but worth being investigated. \mtien{Moreover, since the objective functions of the MCP and robust MCP under GEV models are all submodular, a MILP approach based on submodular cuts would be interesting for future explorations.} 

\section*{Acknowledgments}
We thank the Editor and the three referees for their detailed and
thoughtful comments and suggestions, which substantially improved the
paper. Dr. Tien Mai is  supported by  the National Research
Foundation Singapore and DSO National Laboratories under the AI Singapore Programme (AISG Award No: AISG2-
RP-2020-017)



\bibliographystyle{plainnat_custom}
\bibliography{reference,refs}

\clearpage

\appendix

\begin{center}
{\Huge Appendix}
\end{center}

\section{Proofs}
\label{appd:proofs}
\subsection{Proof of Lemma \ref{lm:lm1}}
\begin{proof}
We will make use of the assumption that  $\cV_i$ can be defined by a set of constraints $\{g^i_t(\bv^i) \leq 0;\; t = 1,\ldots,T\}$. The Lagrangian expression of the minimization problem is 
\[
\cL(\bx,\bv^i,\bld) = G^i(\bY(\bv^{i}) \circ \bx) + \sum_{t\in[T]} \lambda_t g^i_t(\bv^i)
\]
Let $\bld^*$  be the saddle point of Lagrangian (solution that minimize $G^i(\bY(\bv^{i}) \circ \bx)$ subjects to the constraints), the Envelop theorem \citep{Mirrlees1971exploration} implies that 
\begin{align}
 \frac{\partial \phi^i(\bx)}{\partial x_j} &= \left.\frac{\partial \cL(\bx,\bv^i,\bld)}{\partial x_j} \right\rvert_{{(\bv^i,\bld) = (\bv^{i*},\bld^*)}} \nonumber \\
 &= \frac{\partial G^i(\bY(\bv^{i*}) \circ \bx)}{\partial x_j} +  \left.\frac{\partial \left( \sum_{t=1}^T \lambda_t g^t(\bv^i) \right)}{\partial x_j} \right\rvert_{{(\bv^i,\bld) = (\bv^{i*},\bld^*)}} \nonumber\\
 &\stackrel{(a)}{=} \frac{\partial G^i(\bY(\bv^{i*}) \circ \bx)}{\partial x_j},
\end{align}
where $(a)$ is due to the fact that  $\sum_{t=1}^T \lambda_t g^t(\bv^i) = 0$ (complementary slackness from the KKT conditions) does not involve $\bx$. 
We  take the second derivative of $\phi^i(\bx)$ w.r.t. $x_k$ we have
\begin{align}
   \frac{\partial^2 \phi^i(\bx)}{\partial x_j\partial x_k} &= \frac{\partial^2 \cL(\bx,\bv^{i*},\bld^*)}{\partial x_j \partial x_k}  + \sum_{j\in [m]}\frac{\partial^2 \cL(\bx,\bv^{i}(\bx),\bld^*)}{\partial x_j \partial \bv^{i}_j(\bx)} \frac{\partial v^i_j(\bx)}{\partial x_k} \nonumber\\
   &+ \sum_{t\in [T]}\frac{\partial^2 \cL(\bx,\bv^{i*},\bld(\bx))}{\partial x_j \partial \lambda_t(\bx)} \frac{\partial \lambda_t(\bx)}{\partial x_k},\label{eq:lm1-eq1}
\end{align}
where $(\bv^{i}(\bx),\bld(\bx))$ is the saddle point of the Lagrangian  as functions of  $\bx\in[0,1]^m$. We know that, for any $\bx \in [0,1]^m$ and for any $j\in[m], t\in[T]$, the KKT conditions imply that 
\begin{align}
    \frac{\partial \cL(\bx,\bv^{i}(\bx),\bld(\bx))}{\partial v^i_j(\bx) } &= 0; \label{eq:lm1-eq2} 
\end{align}  
and $\lambda_t(\bx) g^i_t(\bv^i(\bx)) = 0$. There are two cases to consider here.  If $\lambda_t(\bx) = 0$ then $\partial \lambda_t(\bx)/\partial x_k = 0$ and if  
$g^i_t(\bv^i(\bx)) = 0$ then  
\[
\left.\frac{\partial \cL(\bx,\bv^{i*},\bld)}{\partial \lambda_t} \right\rvert_{{\lambda = \lambda(\bx)}} = \left. \left(  \frac{\partial G^i(\bY(\bv^{i*}) \circ \bx)}{\partial \lambda_t} +  g^i_t(\bv^{i*}) \right)\right\rvert_{{\lambda = \lambda(\bx)}} =0.
\]
Thus  
\[
\left.\frac{\partial^2 \cL(\bx,\bv^i,\bld)}{\partial \lambda_t \partial x_j} \right\rvert_{{\lambda = \lambda(\bx)}} = 0. 
\]
Combine the two cases, we have
\begin{equation}
\label{eq:lm1-eq3}
 \sum_{t\in [T]}\frac{\partial^2 \cL(\bx,\bv^{i*},\bld(\bx))}{\partial x_j \partial \lambda_t(\bx)} \frac{\partial \lambda_t(\bx)}{\partial x_k} = 0.
\end{equation}
Combine \eqref{eq:lm1-eq1}, \eqref{eq:lm1-eq2}  and \eqref{eq:lm1-eq3} we have 
\begin{align}
     \frac{\partial^2 \phi^i(\bx)}{\partial x_j\partial x_k} &= \frac{\partial^2 \cL(\bx,\bv^{i*},\bld^*)}{\partial x_j \partial x_k} \nonumber \\
     &=\frac{\partial^2 G^i(\bY(\bv^{i*}) \circ \bx)}{\partial x_j \partial x_k} +  \left.\frac{\partial^2 \left( \sum_{t=1}^T \lambda_t g^t(\bv^i) \right)}{\partial x_j \partial x_k} \right\rvert_{{(\bv^i,\bld) = (\bv^{i*},\bld^*)}} \nonumber\\
     &=\frac{\partial^2 G^i(\bY(\bv^{i*}) \circ \bx)}{\partial x_j \partial x_k}.
\end{align}
We complete the proof.
\end{proof}

\subsection{Proof of Proposition \ref{prp:equivalent-general}}
\begin{proof}
Recall that $Y^i_j(v^i_j) = e^{v^i_j}$, then by dividing the numerator and denominator of each fraction by $e^{v^i_0}$, we write the objective function of 
\eqref{prob:general-1}  as
\begin{align}
     \sum_{i\in I}q_i \frac{\sum_{j\in [m]}Y^i_j(v^i_j - v^i_0) \partial G^i_j(\bY(\bv^i)|S)}{1+ \frac{1}{e^{v^i_0}}G^i(\bY(\bv^i)|S)}.
\end{align}
Moreover, from  Property (ii) of Remark \ref{prp:CPGF} and (ii) of Proposition \ref{prp:CPGF} we have 
\begin{align}
\partial G^i_j(\bY(\bv^i)|S) &= \partial G^i_j(\bY(\bv^i)/e^{v^i_0}|S) = \partial G^i_j(\bY(\bv^i -  v^i_0 \bbe)|S) \nonumber\\
\frac{1}{e^{v^i_0}}G^i(\bY(\bv^i)|S)&=G^i(\bY(\bv^i)/e^{v^i_0}|S) = G^i(\bY(\bv^i -  v^i_0 \bbe)|S)\nonumber, 
\end{align}
 where $\bbe$ is an all-ones vector of size $m$. 
Thus, the inner minimization problem of \eqref{prob:general-1} is equivalent to
\begin{equation}
\label{prob:general-2}
 \min_{\substack{(v^i_0,\bv^i) \in \cV_i\\ \forall i\in I}}\left\{ \sum_{i\in I}q_i \frac{\sum_{j\in [m]}Y^i_j(\widetilde{v}^i_j) \partial G^i_j(\bY(\widetilde{\bv}^i)|S)}{1 + G^i(\bY(\widetilde{\bv}^i)|S)} \right\}.    
\end{equation}
where $\widetilde{\bv}^i$ is a vector  of size $m$ with entries $\widetilde{v}^i_j = v^i_j - v^i_0$, for all $i\in I$. Now, by defining new uncertainty sets $\widetilde{\cV}_i$, $\forall i\in I$, we can  write \eqref{prob:general-2} equivalently as 
\begin{equation}
\nonumber
 \min_{\substack{\widetilde{\bv}^i \in \widetilde{\cV}_i\\ \forall i\in I}}\left\{ \sum_{i\in I}q_i \frac{\sum_{j\in [m]}Y^i_j(\widetilde{v}^i_j) \partial G^i_j(\bY(\widetilde{\bv}^i)|S)}{1 + G^i(\bY(\widetilde{\bv}^i)|S)} \right\},    
\end{equation}
as desired. 
\end{proof}

\subsection{Proof of Proposition \ref{prp:convexity-G} }
\begin{proof}
It is more convenient to use the binary representation $G^i(\bY(\bv^i)\circ \bx)$ to prove the claim. To prove the convexity, we take the second derivatives of $G^i(\bY(\bv^i)\circ \bx)$ w.r.t. $\bv^i$ and prove that the Hessian matrix is positive definite. Let $\rho(\bv^i) = G^i(\bY(\bv^i)\circ \bx)$. We take the first and second order derivatives of $\rho(\bv^i)$ with respect to $\bv^i$ and  have 
\begin{align}
    \frac{\partial \rho(\bv^i)}{\partial v^i_j} &= \frac{\partial G^i(\bY(\bv^i)\circ \bx)}{\partial v^i_j}  = (x_jY_j)  \partial G^i_j(\bY(\bv^i)\circ \bx),\: \forall j\in [m]\nonumber \\
        \frac{\partial^2 \rho(\bv^i)}{\partial v^i_j \partial v^i_k} &=  (x_jY_j) (x_kY_k)  \partial G^i_{jk}(\bY(\bv^i)\circ \bx),\: \forall j,k\in[m], j\neq k \nonumber\\
        \frac{\partial^2 \rho(\bv^i)}{\partial v^i_j \partial v^i_j} &=  (x_jY_j)   \partial G^i_{j}(\bY(\bv^i)\circ \bx) +(x_jY_j)^2   \partial G^i_{jj}(\bY(\bv^i)\circ \bx),\: \forall j\in[m]. \nonumber
\end{align}
So, we define $\overline{\bY} \in \bbR^m$ such that $\overline{Y}_j = e^{v^i_j} x_j$ and  write
\[
\nabla^2 \rho(\bv^i) = \dg(\overline{\bY}) \nabla^2G^i(\overline{\bY})\dg(\overline{\bY}) +\dg(\nabla G^i(\overline{\bY}) \circ \overline{\bY})
\]
where $\nabla^2G^i(\overline{\bY})$ is a matrix of size ($m\times m$) with entries  $\partial G^i_{jk}(\overline{\bY})$, for all $j,k\in[m]$, and $\dg(\overline{\bY})$ is the square diagonal matrix with the elements of vector $\overline{\bY}$ on the main diagonal. We see that $\dg(\nabla G^i(\overline{\bY}) \circ \overline{\bY})$ is positive definite. Moreover, $\dg(\overline{\bY}) \nabla^2G^i(\overline{\bY})\dg(\overline{\bY})$ is symmetric and its $(j,k)$-th element  is $\overline{Y}_j\overline{Y}_k \partial G^i_{jk}(\overline{\bY})$. For $j\neq k$ we have  
$ \partial G^i_{jk}(\overline{\bY}) \leq 0$ (Property \textit{(iv)} of Remark \ref{prp:CPGF}), thus all the off-diagonal entries of the matrix are non-positive. Moreover, from Property \textit{(iii)} of Proposition \ref{prp:CPGF-new}, we see that $\sum_{k\in [m]}\overline{Y}_j\overline{Y}_k \partial G^i_{jk}(\overline{\bY}) = 0$ for any $j\in[m]$, thus each row of the matrix sums up to zero. Using  Theorem A.6 of \cite{DeKlerk2006aspects}, we see that   $\dg(\overline{\bY}) \nabla^2G^i(\overline{\bY})\dg(\overline{\bY})$ is positive semi-definite. Since  $\dg(\nabla G^i(\overline{\bY}) \circ \overline{\bY})$ is positive definite, $\nabla^2 \rho(\bv^i)$ is positive definite,  implying that $\rho(\bv^i)$ is strictly convex in $\bv^i$, as desired.
\end{proof}

\section{\mtien{Comparison with   \cite{mehmanchi2020robust}}}\label{sec:comp-meh}
\cite{mehmanchi2020robust} study a robust fractional 0-1 program based on the  uncertainty structure introduced by \citep{bertsimas2004robust,bertsimas2004price}, which can be well applied to the MCP under MNL. The work of \cite{mehmanchi2020robust} differs from our robust MCP under MNL by the fact that our methods work with any convex uncertainty sets while \cite{mehmanchi2020robust} employ rectangular uncertainty sets where each unknown coefficient (i.e., utility in the MCP context)   lies in a symmetric interval centered on a nominal value. Thus, our algorithms (i.e., both the local search and outer-approximation) can work well with their uncertainty setting. In this section, we consider a robust MCP under the \cite{mehmanchi2020robust}'s uncertainty setting and make a numerical comparison between our algorithms and  their solution approach which is based on a MILP reformulation. 

Under \cite{mehmanchi2020robust}'s uncertainty setting, each uncertainty set $\cV_i$ is defined as  $\cV_i = \{\bV^i| V_{ij} \in [\overline{V}_{ij}-d_{ij};\overline{V}_{ij}+d_{ij}], ~|S_i(\bV)|\leq \Gamma_i\}$ where $\overline{V}_{ij}$ are nominal utilities, coefficients $d_{ij}\geq 0$ denote potential deviation from the nominal utilities, $S_i(\bV)$ is the set of indices of the uncertain parameters $\bV$ whose values are different from the nominal values $\overline{\bV}$. The constraints $S_i(\bV)\leq \Gamma_i$ imply that there are at most $\Gamma_i$ unknown coefficients taking values different from their nominal values.  
The robust MCP under MNL can be formulated as:
\begin{align}
      \max_{\bx \in \{0,1\}^m}& \left\{f^{\WC}(\bx) = \sum_{i\in I} \min_{\bV^i \in \cV_i} \left\{q_i -  \sum_{i\in I}\frac{q_i}{1+ {}\sum_{{j}}{V_{ij}}x_{j} }\right\} \right\}\label{prob:ro2}\tag{\sf RO*} \\
     \text{subject to} & \sum_{j\in [m] }x_{j} = C. \nonumber
\end{align}
Since the inner problem is a minimization problem, we want to maximize the number of times that  $V_{ij}$ is equal to its lower bound. Thus, we can use additional  binary variables $u_{ij}\in\{0,1\},~\forall i\in [I], j\in [m]$ to  write the MCP problem as:
\begin{align}
      \max_{\bx \{0,1\}^m} &\left\{f^{\WC}(\bx) = \sum_{i\in I} q_i -  \sum_{i\in I}\frac{q_i}{1 + \sum_{j} \overline{V}_{ij}x_{j} - \max_{\substack{u_{ij},~\forall j\in [m]\\\sum_{j}u_{ij}\leq \Gamma_i}}\left\{ \sum_{j} d_{ij}x_j u_{ij}\right\}} \right\}\nonumber\\
     \text{subject to} & \sum_{j\in [m] }x_{j} = C.\nonumber
\end{align}
To reformulate the above fractional program as a MILP,  we let $$y_i =   1/\left(1 + \sum_{j} \overline{V}_{ij}x_{j} - \max_{\substack{u_{ij}\in\{0,1\}\\\sum_{j}u_{ij}\leq \Gamma_i}}\left\{ \sum_{j} d_{ij}x_j u_{ij}\right\}\right)$$
and denote $z_{ij} = y_i x_j$. We further linearize these bi-linear terms using  McCormick inequalities \citep{mccormick1976computability} to obtain:
\begin{align}
     \max_{\bx \in \{0,1\}^m} & \left\{f^{\WC}(\bx) = \sum_{i\in I} q_i -  \sum_{i\in I} q_iy_i \right\}\nonumber \\
     \text{subject to} & \sum_{j\in [m] }x_{j} = C\nonumber \\
     &y_j + \sum_{j} \overline{V}_{ij}z_{ij} - \max_{\substack{u_{ij}\in\{0,1\}\\
     \sum_{j}u_{ij}\leq \Gamma_i}}\left\{ \sum_{j} d_{ij}z_{ij} u_{ij}\right\} \geq 1 \nonumber \\
      & z_{ij} \geq y_{i}- \overline{y}_i(1-x_{j}) & \forall i\in I, \forall j\in [m] \nonumber\\
    & z_{ij}\leq y_{i} +\underline{y}_{i}(x_{j}-1) & \forall i\in I, \forall j\in [m]\nonumber\\
    & z_{ij} \leq \overline{y}_{i}x_{j} & \forall i\in I, \forall j\in [m]\nonumber\\
    &  z_{ij} \geq \underline{y}_{i}x_{j} & \forall i\in I, \forall j\in [m]\nonumber\\
    &y_j\geq 0, z_{ij}\geq 0,~ \forall i\in I, \forall j\in [m], \nonumber
\end{align}
where  $\overline{y}_{i},~\underline{y}_{i}$, $\forall i\in I$, are some upper and lower bounds of $y_i$, respectively. Using  Equation (9) of \cite{mehmanchi2020robust}, we take the dual of the inner maximization problem and formulate the robust MCP as the following MILP: 
\begin{align}
     \max_{\bx} &\left\{f^{\WC}(\bx) = \sum_{i\in I} q_i -  \sum_{i\in I} q_iy_i \right\}\label{prob:milp*}\tag{\sf MILP*} \\
     \text{subject to} & \sum_{j\in [m] }x_{j} = C\nonumber \\
     &y_j + \sum_{j} \overline{V}_{ij}z_{ij} - 1 \geq \Gamma_i\alpha + \sum_j p_j \nonumber \\
     &p_j+\alpha \geq d_{ij}z_{ij} \nonumber \\
      & z_{ij}\leq y_{i} +\underline{y}_{i}(x_{j}-1) & \forall i\in I, \forall j\in [m]\nonumber\\
    & z_{ij} \leq \overline{y}_{i}x_{j} & \forall i\in I, \forall j\in [m]\nonumber\\
    &  z_{ij} \geq \underline{y}_{i}x_{j} & \forall i\in I, \forall j\in [m]\nonumber\\
    &\alpha, p_j, y_j, z_{ij}\geq 0,&\forall i\in I, \forall j\in [m], \nonumber
\end{align}
where $\alpha$ and $p_j,~j\in [m]$ are dual variables of the inner maximization problem.

To conduct the experiment, we take  instances from the three datasets HM, ORlib, and NYC. We choose the  deviation coefficients as $d_{ij} = 0.5\overline{V}_{ij}$, $\forall i\in I, j\in [m]$ and  the level of uncertainty $\Gamma_{i} \in [1,2,3,4,5]$ and  $C = 5$.
We set a time limit of 600 seconds for all the methods and  select lower and uppers bounds for $y_i$ as $\overline{y}_i = 1$ and $\underline{y}_i = 0$, for all $i\in I$,  similarly to the setup in  \cite{mehmanchi2020robust}.  
We will compare \eqref{prob:milp*} against our local search algorithm (i.e. Algorithm \ref{algo:local-search}), noting that the MOA can be used as well but it is generally less efficient than the local search algorithm, in terms of both computing time and solution quality. For each step of LS, we compute the worst-case objective value by solving the inner maximization $\max_{\substack{u_{ij}\in\{0,1\}\\
     \sum_{j}u_{ij}\leq \Gamma_i}}\left\{ \sum_{j} d_{ij}x_{j} u_{ij}\right\}$ ($\bx$ is fixed). This can be done efficiently by relaxing the binary variables $u_{ij}$ and using CPLEX to solve the resulting linear program.
Here, we only focus on comparison in terms of solving the robust problem, as the value and price of robustness under this  uncertainty setting were intensively assessed in previous work \citep{mehmanchi2020robust,bertsimas2004price}.   

Table \ref{tab:compare_Meh} below reports the performance of \eqref{prob:milp*} solved by CPLEX and our LS algorithm, in terms of running time and the number of times each method returns better objective values within the time budget. The results generally show that LS outperforms \eqref{prob:milp*} in terms of solution quality and is also faster, especially for large-sized instances.

\begin{table}[htb]
\centering
\begin{tabular}{l|r|r|r|r|r|r}
\multicolumn{1}{c|}{\multirow{2}{*}{Dataset}} & \multicolumn{1}{c|}{\multirow{2}{*}{$|I|$}} & \multicolumn{1}{c|}{\multirow{2}{*}{$m$}} & \multicolumn{2}{l|}{\begin{tabular}[c]{@{}l@{}}\# instances with\\~better objectives\end{tabular}} & \multicolumn{2}{c}{\begin{tabular}[c]{@{}c@{}}Average \\running time (seconds)\end{tabular}}  \\ 
\cline{4-7}
\multicolumn{1}{c|}{}& \multicolumn{1}{c|}{}  & \multicolumn{1}{c|}{}      & \multicolumn{1}{l|}{MILP*} & \multicolumn{1}{l|}{LS (ours)} & \multicolumn{1}{l|}{MILP*} & \multicolumn{1}{l}{LS (ours)}  \\ 
\hline
HM14& 50& 25& 5& 5& 0.4  & 12.1      \\
HM14& 50& 50& 5& 5& 1.1  & 9.6       \\
HM14& 50& 100& 5& 5& 4.3  & 12.7      \\
HM14& 100 & 25& 5& 5& 1.0  & 11.7      \\
HM14& 100 & 50& 5& 5& 3.8  & 11.7      \\
HM14& 100 & 100& 5& 5& 13.5 & 13.6      \\
HM14& 200 & 25& 5& 5& 4.8  & 14.9      \\
HM14& 200 & 50& 5& 4& 22.4 & 16.3      \\
HM14& 200 & 100& 5& 3& 152.4& 21.7      \\
HM14& 400 & 25& 5& 5& 14.9 & 15.5      \\
HM14& 400 & 50& 5& 5& 78.4 & 18.9      \\
HM14& 400 & 100& 3& 4& 600& 24.0      \\
HM14& 800 & 25& 5& 5& 136.2& 17.9      \\
HM14& 800 & 50& 2& 4& 600& 15.1      \\
HM14& 800 & 100& 2& 3& 600& 22.7      \\ 
\hline
ORlib& 50& 25& 5& 4& 27.1 & 11.8      \\
ORlib& 50& 25& 5& 4& 27.1 & 10.2      \\
ORlib& 50& 25& 5& 5& 25.4 & 11.5      \\
ORlib& 50& 25& 5& 5& 25.6 & 10.2      \\
ORlib& 50& 50& 5& 4& 600& 12.4      \\
ORlib& 50& 50& 5& 4& 600& 13.7      \\
ORlib& 50& 50& 5& 4& 600& 12.5      \\
ORlib& 50& 50& 5& 5& 600& 13.8      \\
ORlib& 1000& 100& 0& 5& 600& 16.3      \\
ORlib& 1000& 100& 0& 5& 600& 14.3      \\
ORlib& 1000& 100& 0& 5& 600& 15.0      \\ 
\hline
NYC & 82341          & 59& 0& 5& \multicolumn{1}{l|}{} &     \\ 
\hline
Average & \multicolumn{1}{l}{}   & \multicolumn{1}{l|}{}      & 4.0      & 4.6 & \multicolumn{2}{l}{}
\end{tabular}
\caption{Comparison between the MILP approach proposed in \cite{mehmanchi2020robust} and Algorithm \ref{algo:local-search} for the  MCP under MNL.}
\label{tab:compare_Meh}
\end{table}

\section{Additional Experiments}
\label{appd:exp}

\subsection{Comparing across the MNL and Nested Logit Models}

We will assess the performance of the  RO approach across the MNL and nested logit instances. The aim is to see how the RO performs, as compared to the  other approaches, when the correlation structure of the choice model changes. To this end, we will provide results under the following two settings. First, we  solve the robust MCP problem under MNL and nested logit by Algorithm \ref{algo:MOA} and \ref{algo:local-search} and compare the obtained solutions  by injecting them into the corresponding nested logit instances and comparing  the corresponding distributions of the objective values, in a similar way as in the above sections. The aim is to evaluate gains that we can get if the choice model is well-specified, under our robust settings.   
Second, we  solve MNL instances by the RO, DET1, DET2 and SA approaches. We then take the obtained solutions  to test on the corresponding nested logit instances in the same manner as in the previous sections. In other words, we solve the robust MCP problem under MNL but test the solutions obtained on  nested logit instances.
By doing this, we aim to explore how different approaches protect us from worst-case scenarios when the choice model is misspecified.

{For the first setting, we plot in Figure \ref{fig:mnl_vs_nested_1} the distributions of the 2000 samples of the objective values given by MNL and nested-logit solutions with $\epsilon \in \{0.02,0.04,0.08, 0.4,0.5,0.6\}$, for instances of size $|I|=100$ and $m=50$.
The figure clearly shows that the nested-logit solutions always return better objective values than those given by the MNL solutions. Moreover, the difference seems smaller as $\epsilon$ increases. }

\begin{figure}[htb]
  \centering
  \includegraphics[width=1.0\linewidth]{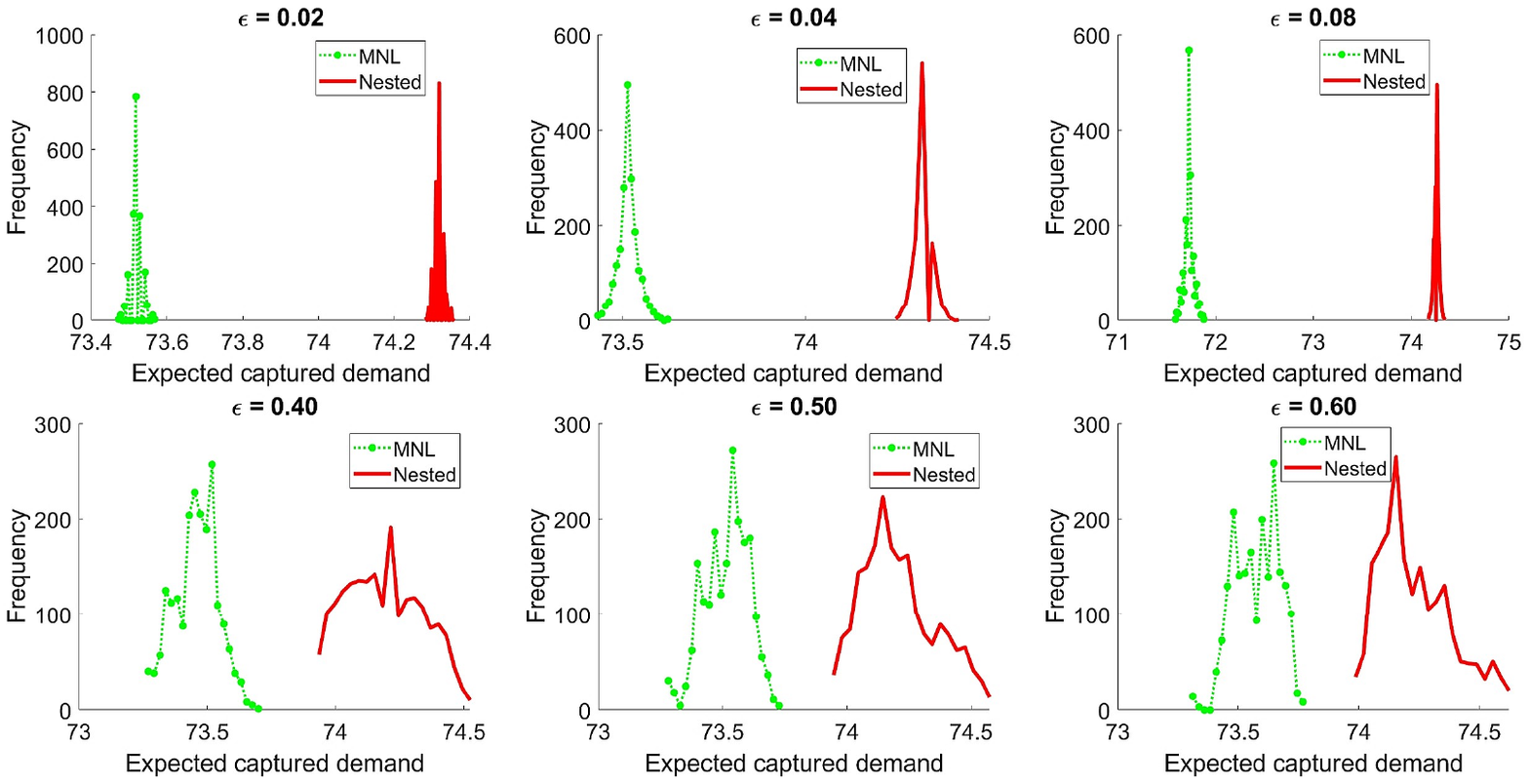}  
  \caption{Comparison of the distributions of the \textit{nested-logit} objective values given by solutions from the MNL and the nested logit model for instance of $|I|=100$ and $m=50$.}
 \label{fig:mnl_vs_nested_1}
\end{figure}

For the second setting, in Table \ref{table1}, we report the averages of 2\% of the worst-case nested logit objective values for each $\epsilon$ value. It is quite clear that RO is not able to retain the same advantages as in the cases considered above, especially for some large-scale instances where the nested correlation structures become more complex. The SA and DET2 approaches seem to provide better protection in this case, suggesting that more investigations would be needed if the correlation structure of the choice model is not well specified. 

In summary, if  the choice model is well-specified in terms of correlation structure,  our results show gains from our robust model in protecting decision-makers from expected user demand
that would be too low. The histograms given by the robust approach have high peaks, small
variances, and high worst-case values, as compared to their deterministic and
sampling-based counterparts. The sampling-based SA approach can provide some protection as well, but it is much more expensive than the RO. 
If the correlation of the choice model is not well specified, we observe that the RO is limited in maintaining the same advantage, suggesting that more investigations, or possibly, new robust models would be needed to address such a misspecification issue. This is out-of-scope of the paper  and we keep this for future work. 


\begin{table}[]
\centering
\begin{tabular}{l|l|l|l|l|l|l}
$|I|$   & $m$   & $\epsilon$    & RO & DET1 & DET2 & SA\\ \hline
100   & 50  & 0.02 & 73.48& 73.48& 73.85& \textbf{73.87}    \\
100   & 50  & 0.04 & 73.45& 73.45& \textbf{73.82}    & 72.95\\
100   & 50  & 0.08 & 71.62& 73.38& \textbf{73.73}    & 72.91\\
100   & 50  & 0.40 & 73.27& 72.80& 73.13& \textbf{73.53}    \\
100   & 50  & 0.50 & \textbf{73.29}  & 72.78& 73.13& 71.91\\
100   & 50  & 0.60 & \textbf{73.37}  & 72.79& 73.16& 72.55\\
100   & 100 & 0.02 & 69.51& 69.51& \textbf{73.12}    & \textbf{73.12}    \\
100   & 100 & 0.04 & 69.50& 69.50& \textbf{73.09}    & \textbf{73.09}    \\
100   & 100 & 0.08 & 69.47& 69.47& \textbf{73.05}    & \textbf{73.05}    \\
100   & 100 & 0.40 & 70.48& 69.34& 72.78& \textbf{72.94}    \\
100   & 100 & 0.50 & 70.51& 69.36& \textbf{72.81}    & 71.01\\
100   & 100 & 0.60 & 70.55& 69.40& 72.85& \textbf{72.93}    \\
200   & 100 & 0.02 & 140.63          & 140.63 & 140.63 & 140.63 \\
200   & 100 & 0.04 & 140.56          & 140.56 & 140.56 & 140.56 \\
200   & 100 & 0.08 & \textbf{140.41} & 140.38 & 140.38 & 134.13 \\
200   & 100 & 0.40 & \textbf{145.09} & 139.10 & 139.10 & 133.83 \\
200   & 100 & 0.50 & 145.15          & 138.89 & 138.89 & 140.17 \\
200   & 100 & 0.60 & \textbf{140.23} & 138.71 & 138.71 & 134.11 \\
1000  & 100 & 0.02 & 38104.29        & 38266.32          & 38104.29          & 38104.29          \\
1000  & 100 & 0.04 & 38099.74        & 38260.89          & 38099.74          & \textbf{38831.59} \\
1000  & 100 & 0.08 & 38087.11        & 38244.76          & 38087.11          & \textbf{38823.69} \\
1000  & 100 & 0.40 & 38281.47        & 38135.82          & 38008.28          & \textbf{38963.94} \\
1000  & 100 & 0.50 & 38291.25        & 38124.55          & 38003.93          & \textbf{38362.69} \\
1000  & 100 & 0.60 & 38300.33        & 38114.96          & 38002.34          & \textbf{39164.15} \\
82341 & 59  & 0.02 & 68497.10        & \textbf{70338.96} & \textbf{70338.96} & 68497.10          \\
82341 & 59  & 0.04 & 68442.18        & \textbf{70246.72} & \textbf{70246.72} & 68442.18          \\
82341 & 59  & 0.08 & 68345.19        & \textbf{70076.93} & \textbf{70076.93} & 68345.19          \\
82341 & 59  & 0.40 & 67974.98        & 68772.24          & 68772.24          & \textbf{69141.38} \\
82341 & 59  & 0.50 & 67999.14        & \textbf{68579.18} & \textbf{68579.18} & 67999.14          \\
82341 & 59  & 0.60 & 68069.08        & \textbf{68442.86} & \textbf{68442.86} & 68069.08         
\end{tabular}
\caption{{Comparison of {nested logit} objective values  given by solutions obtained by solving MNL instances. }}
\label{table1}
\end{table}

\subsection{Objective Value Distributions}

\paragraph{Instances of  $|I|=100$ and $m=100$. }
Figure \ref{fig:100-100-mnl} shows the histograms given by the four approaches with MNL instances. When $\epsilon$ is small, the histograms are very similar across the four approaches. When $\epsilon\geq 0.4$, in analogy  to the experiments shown in the main part of the paper, the histograms given by the RO solutions have small variance and shorted left tails, indicating the capability of RO to cover too-low expected demand values. 

Figure \ref{fig:100-100-nested} shows comparison results for nested logit instances, where  the histograms from the four approaches are identical for $\epsilon\leq 0.08$. When $\epsilon\geq 0.4$,  the superiority of RO in terms of worst-case protection becomes much clearer. We  note that RO seems to give better protection with nested logit instances, as compared to the MNL instances shown in Figure \ref{fig:100-100-mnl}.

In Figure \ref{fig:100-100-ranks-mnl-nested} we plot the percentile ranks of the RO's worst-case objective values in the distributions  given by the other approaches. We see that the ranks only become significant with $\epsilon>0.4$ for MNL instances and with $\epsilon>0.2$ for nested logit instances. The solutions given by DET1 and DET2 are identical for these testing instances.  
We  see that the percentile ranks are higher for nested logit instances, indicating that the RO would give better protection under the nested logit.     


\begin{figure}[htb]
  \includegraphics[width=\linewidth]{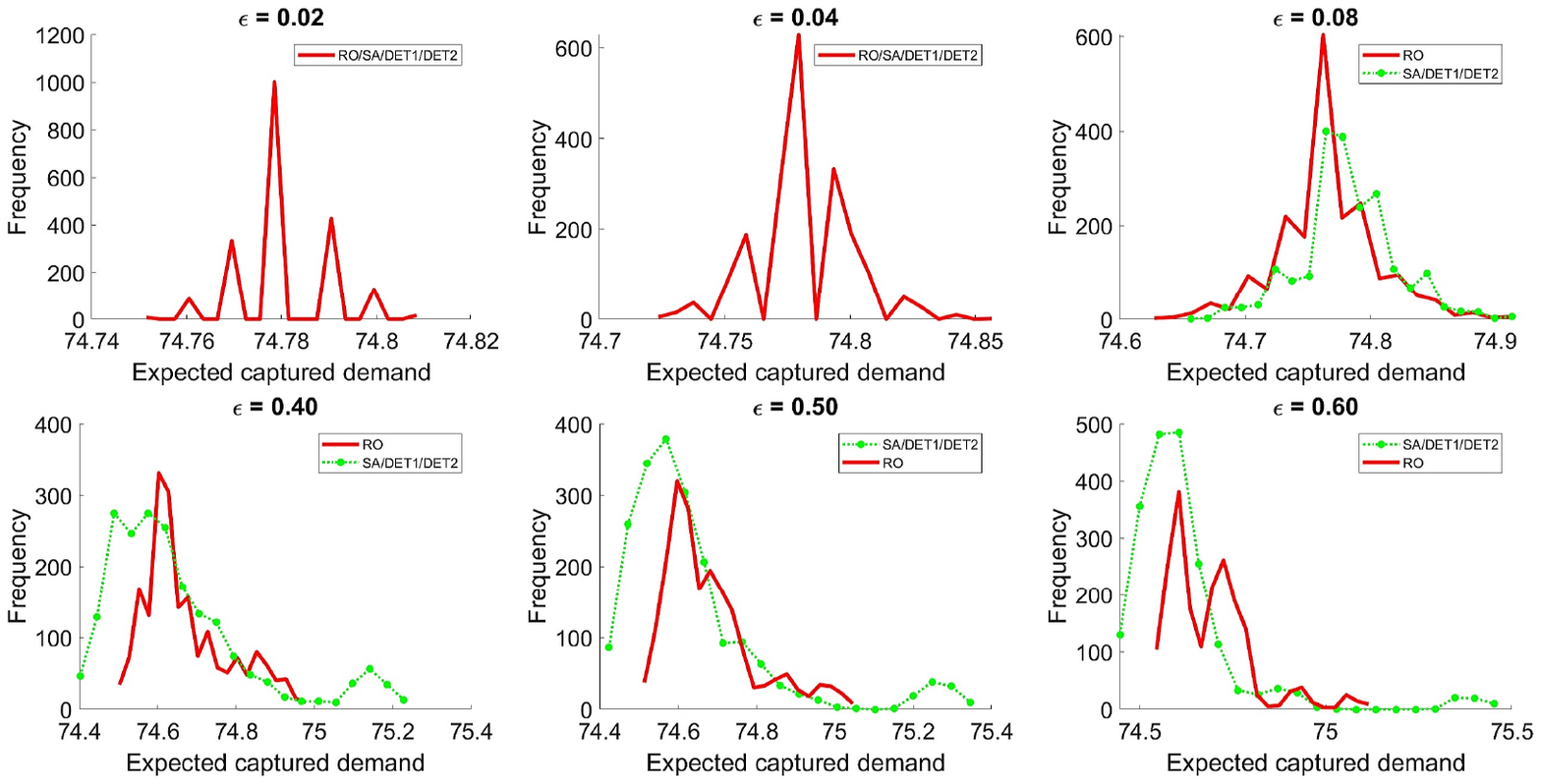}  
  \caption{Comparison results between the distributions of the objective values given by solutions from the RO, SA, DET1, and DET2 approaches, under the MNL choice model and with instances of size $|I|=100$ and $m = 100$.}
  \label{fig:100-100-mnl}
\end{figure}

\begin{figure}[htb]
  \includegraphics[width=\linewidth]{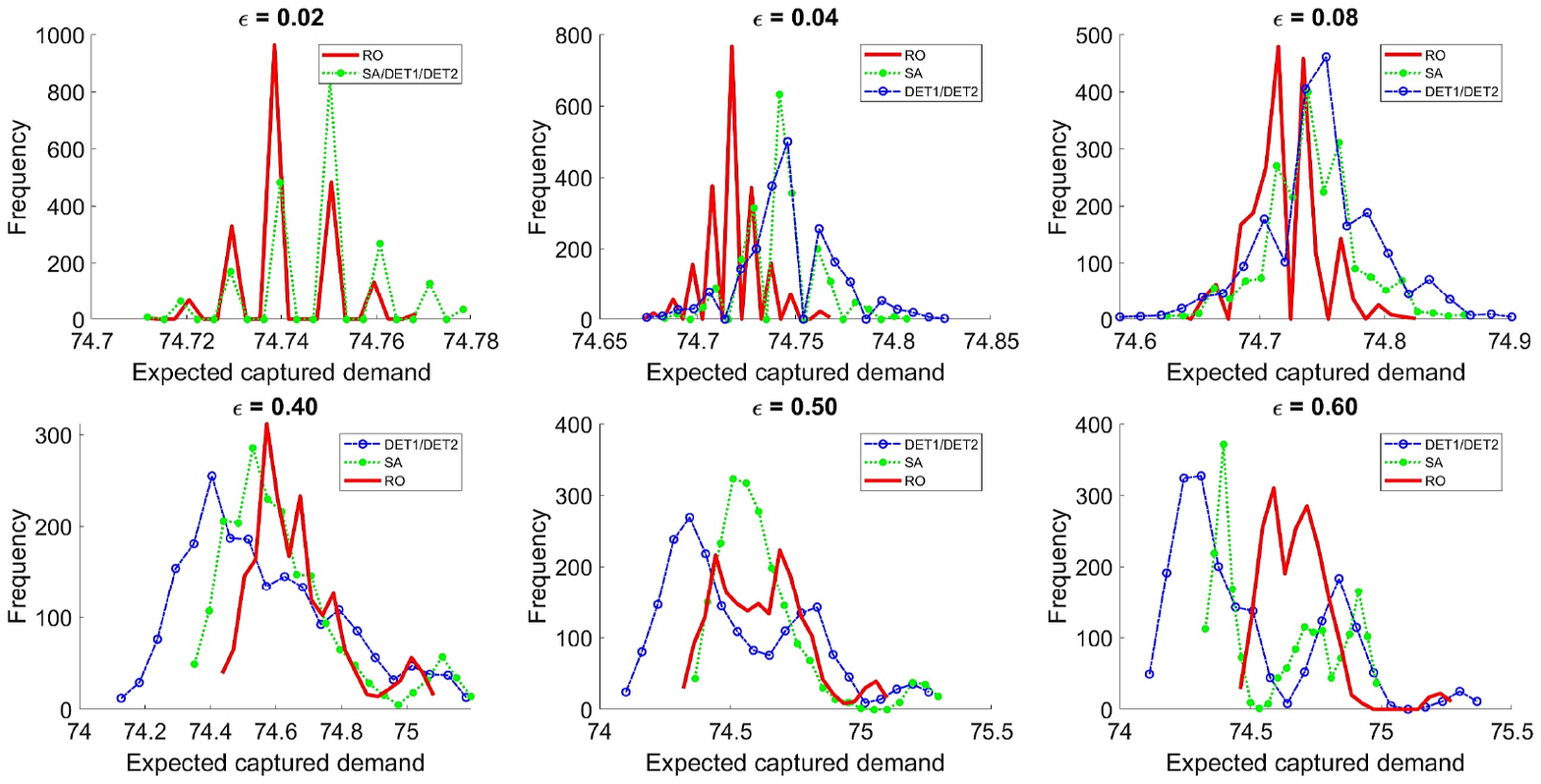}  
  \caption{Comparison of the distributions of the objective values given by solutions from RO, SA, DET1, and DET2 approaches, under the nested choice model and with instances of size $|I|=100$ and $m = 100$.}
  \label{fig:100-100-nested}
\end{figure}

\begin{figure}[htb]
  \centering
  \includegraphics[width=0.8\linewidth]{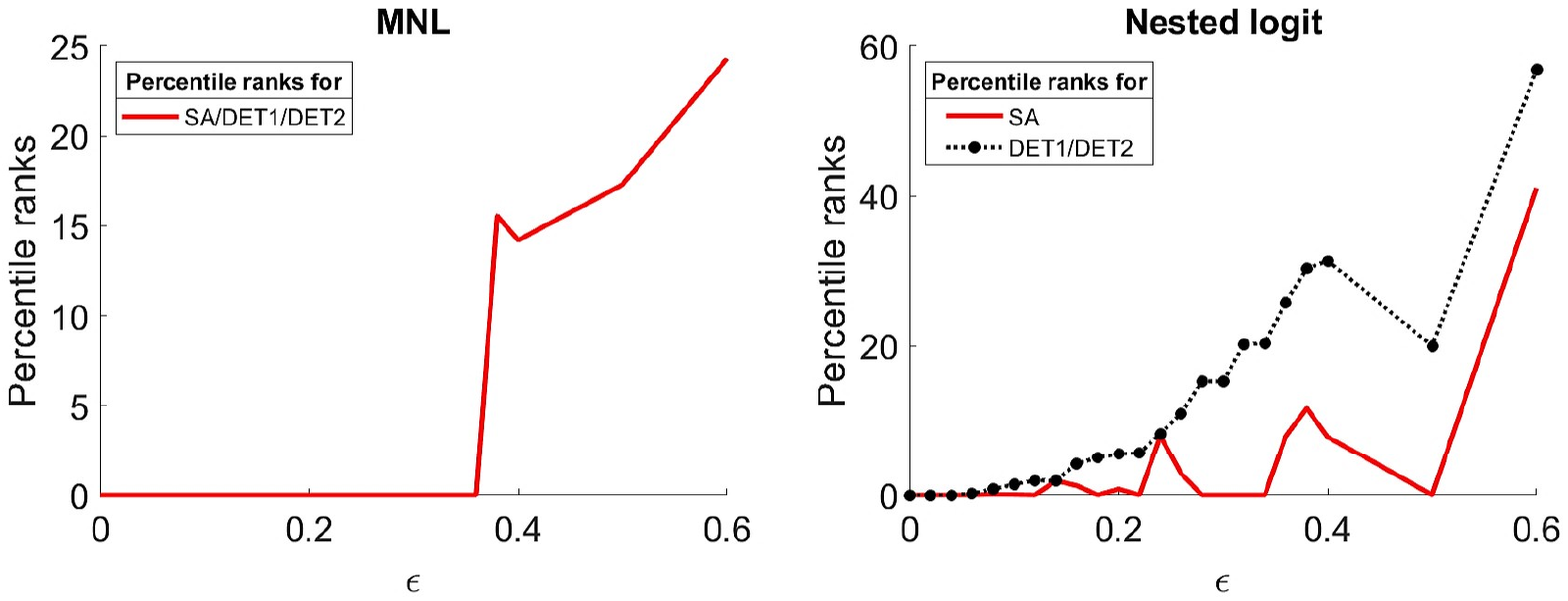}  
  \caption{The percentile ranks of RO worst value in the distributions given by the SA, DET1, and DET2 solutions under the MNL choice model for instance with $|I|=100$ and $m=100$.}
 \label{fig:100-100-ranks-mnl-nested}
\end{figure}


\paragraph{Instances of  $|I|=200$ and $m=100$.}

Figure \ref{fig:200-100-mnl} shows the histograms given by the four approaches for MNL instances, where  the ability of RO to provide protection is clearly shown. When $\epsilon\geq 0.4$, the histograms given by RO always have higher peaks, shorter tails, and higher worst-case objective values. The trade-off of being robust can also be seen. Figure \ref{fig:200-100-nested} shows the histograms  for nested logit instances. Figure \ref{fig:200-100-ranks} shows the percentile ranks of the RO's worst-case objective values. In analogous to the previous experiments, the ability of RO in protecting the decision-maker from low objective values is clearly demonstrated. 

\begin{figure}[htb]
  \includegraphics[width=\linewidth]{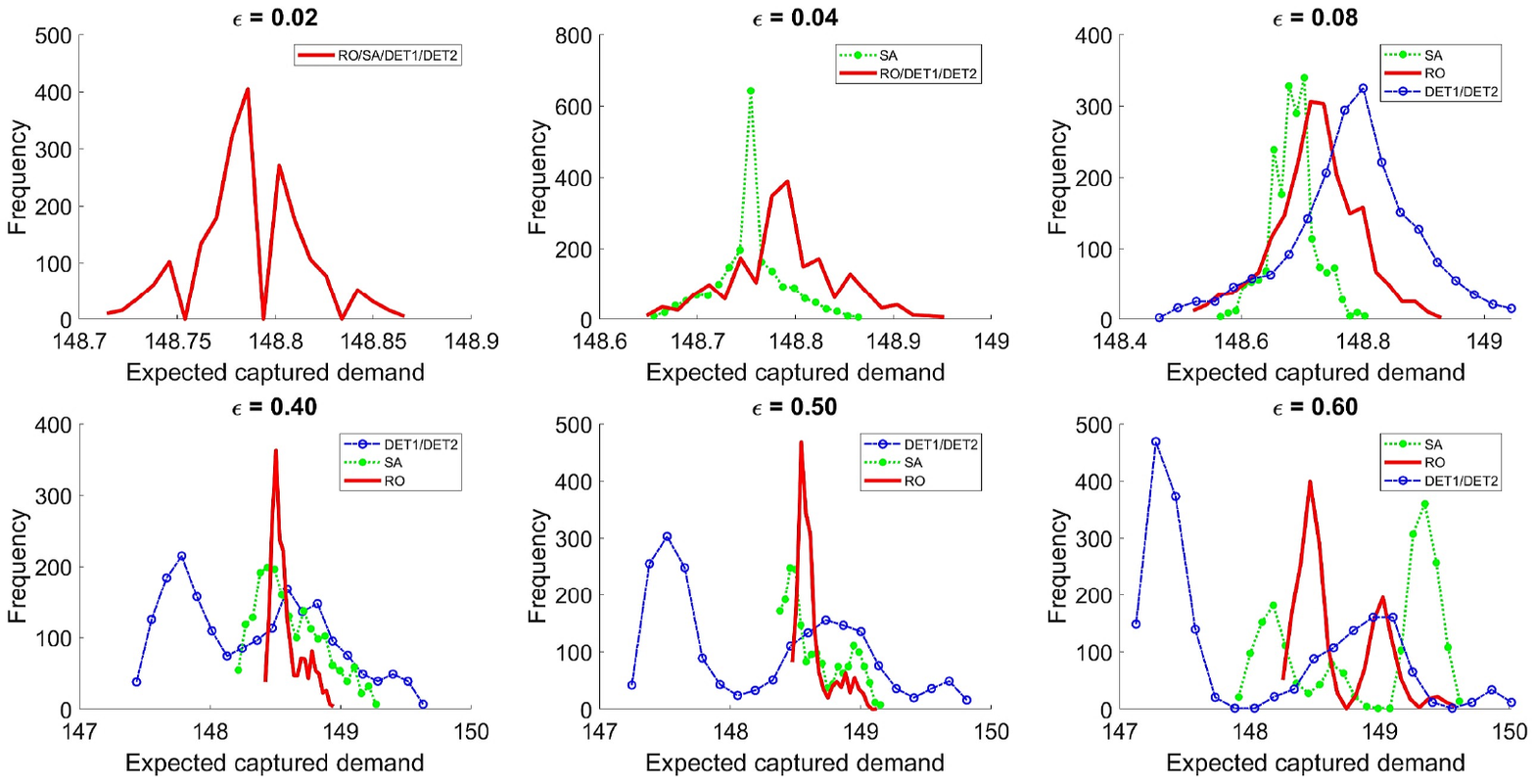}  
  \caption{Comparison between the distributions of the objective values given by solutions from RO, SA, DET1, and DET2 approaches, under the MNL choice model and with instances of size $|I|=200$ and $m = 100$.}
  \label{fig:200-100-mnl}
\end{figure}

\begin{figure}[htb]
  \includegraphics[width=\linewidth]{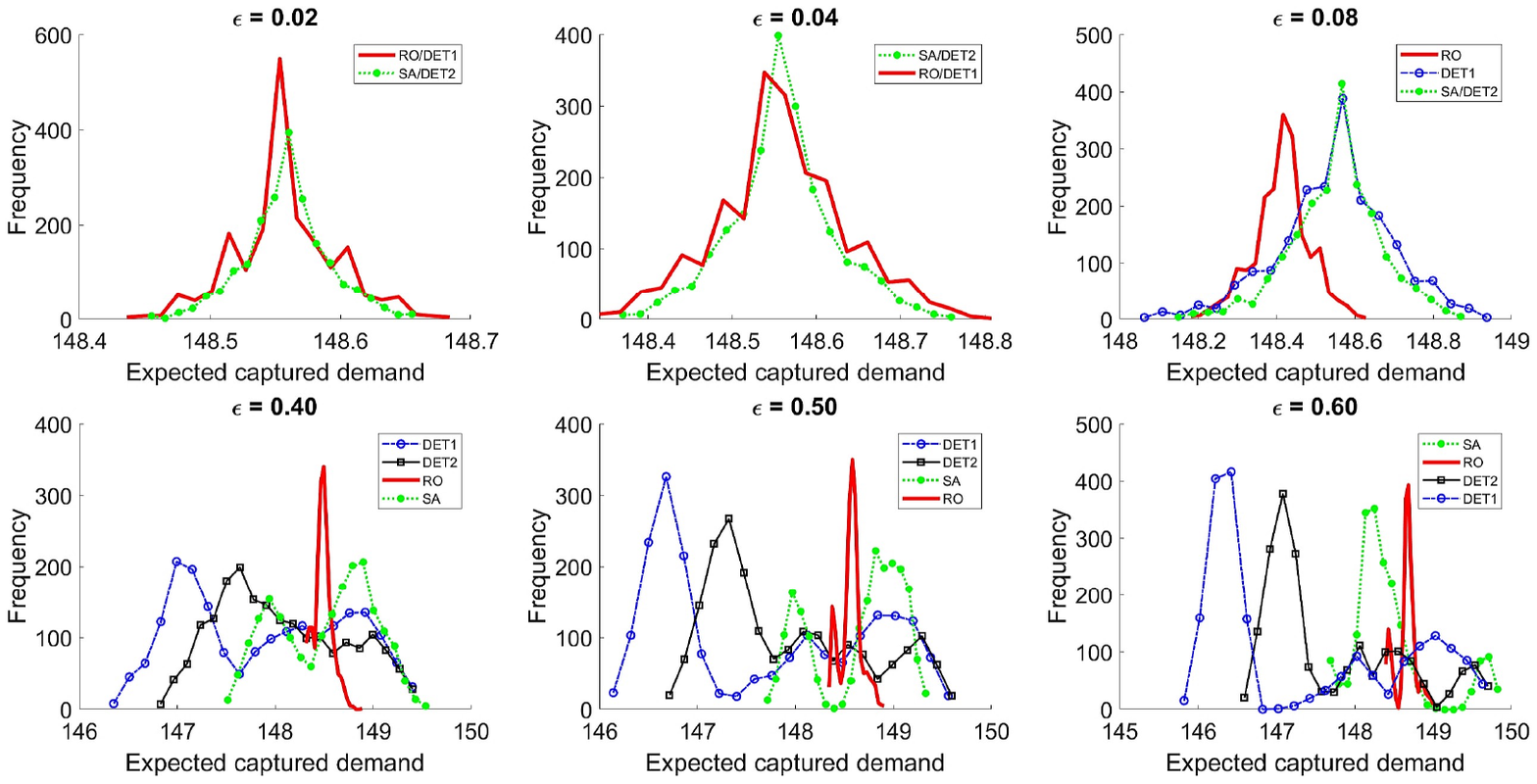}  
  \caption{Comparison between the distributions of the objective values given by solutions from RO, SA, DET1, and DET2 approaches, under the nested logit choice model and with instances of size $|I|=200$ and $m = 100$.}
  \label{fig:200-100-nested}
\end{figure}

\begin{figure}[H]
  \centering
  \includegraphics[width=0.8\linewidth]{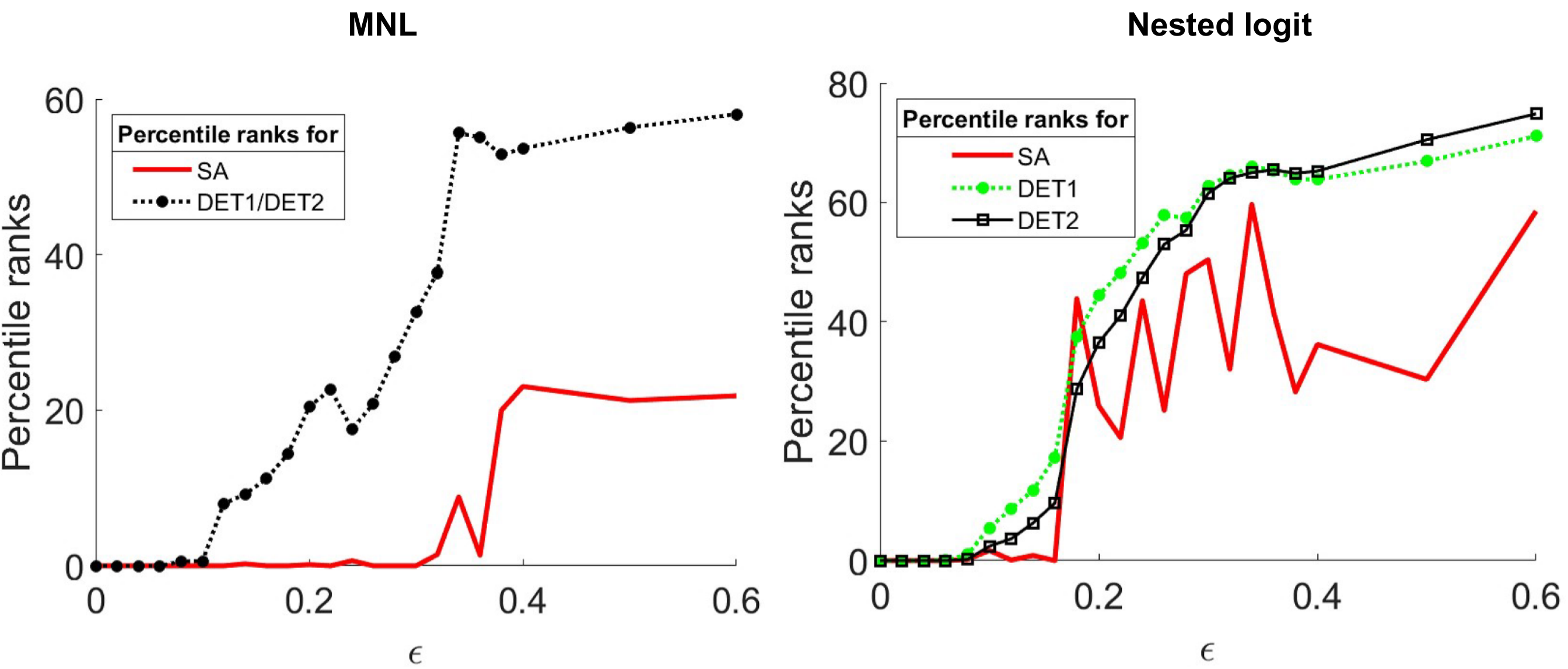}  
  \caption{The percentile ranks of RO worst value in the distributions given by the SA, DET1, and DET2 solutions under the nested logit choice model for instance with $|I|=200$ and $m=100$.}
 \label{fig:200-100-ranks}
\end{figure}


\paragraph{{Instances of $|I|=1000$ and $m=100$}.}
Figure \ref{fig:1000-100-mnl} shows the histograms given by the four approaches under the MNL choice model. The difference between the histograms of the four approaches is not clear when $\epsilon$ is small. With $\epsilon \geq 0.4$, the histograms given by the RO solutions have very small variance, shorted left tails, and  particularly high peaks, as compared to the  other approaches. Figure \ref{fig:1000-100-nested} below shows the histograms given by the  four approaches under the nested logit choice model. The histograms of RO, SA, and DET1 are similar at $\epsilon \in \{0.02, 0.04\}$. When $\epsilon$ becomes larger, RO gives histograms of higher peaks, smaller variances, and shorter tails, as compared to other approaches. It clearly shows the ability of RO in protecting the decision-makers against worse-case situations. In Figure \ref{fig:1000-100-ranks}, we plot the percentile ranks of the RO's worst objective value. We see that the percentile ranks only become significant when $\epsilon >0.26$ for the MNL and $\epsilon > 0.2$ for the nested logit instances. 

\begin{figure}[H]
  \includegraphics[width=\linewidth]{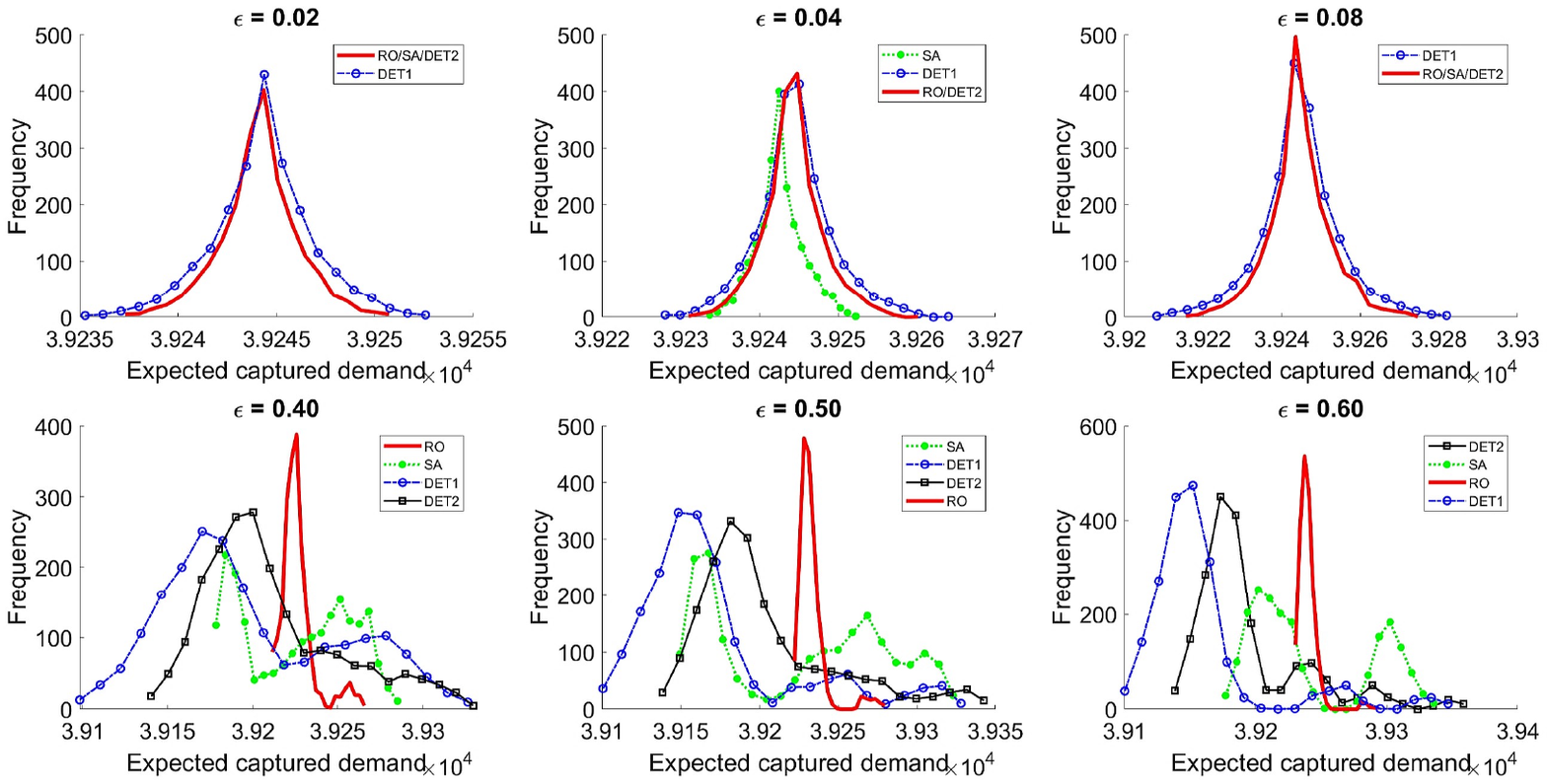}  
  \caption{\textcolor{red}{Comparison between the distributions of the objective values given by solutions from RO, SA, DET1, and DET2 approaches, under the MNL choice model and with instances of size $|I|=1000$ and $m = 100$.}}
  \label{fig:1000-100-mnl}
\end{figure}

\begin{figure}[H]
  \includegraphics[width=\linewidth]{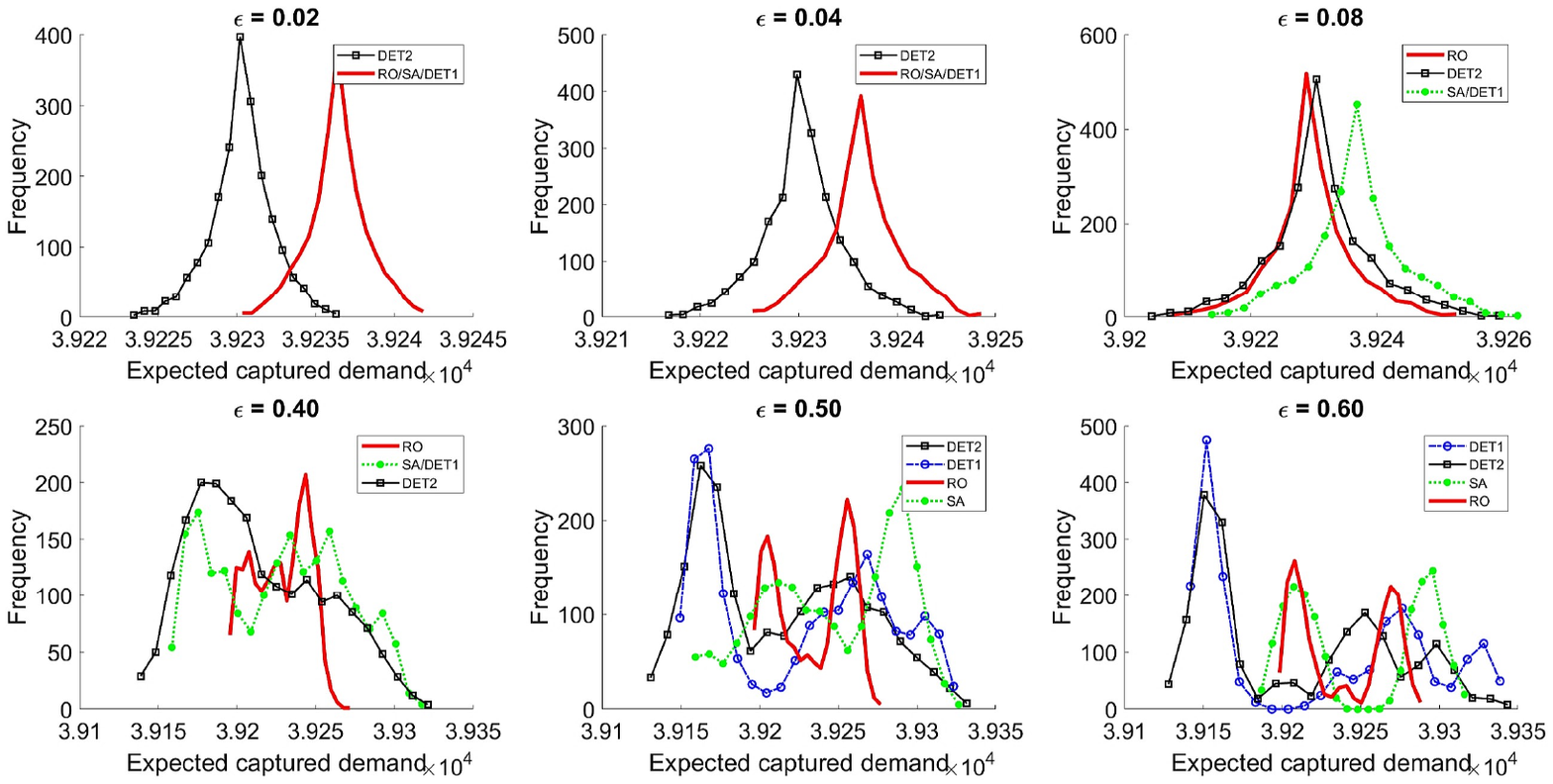}  
  \caption{\textcolor{red}{Comparison between the distributions of the objective values given by solutions from RO, SA, DET1, and DET2 approaches, under the nested logit choice model and with instances of size $|I|=1000$ and $m = 100$.}}
  \label{fig:1000-100-nested}
\end{figure}

\begin{figure}[H]
  \centering
  \includegraphics[width=0.8\linewidth]{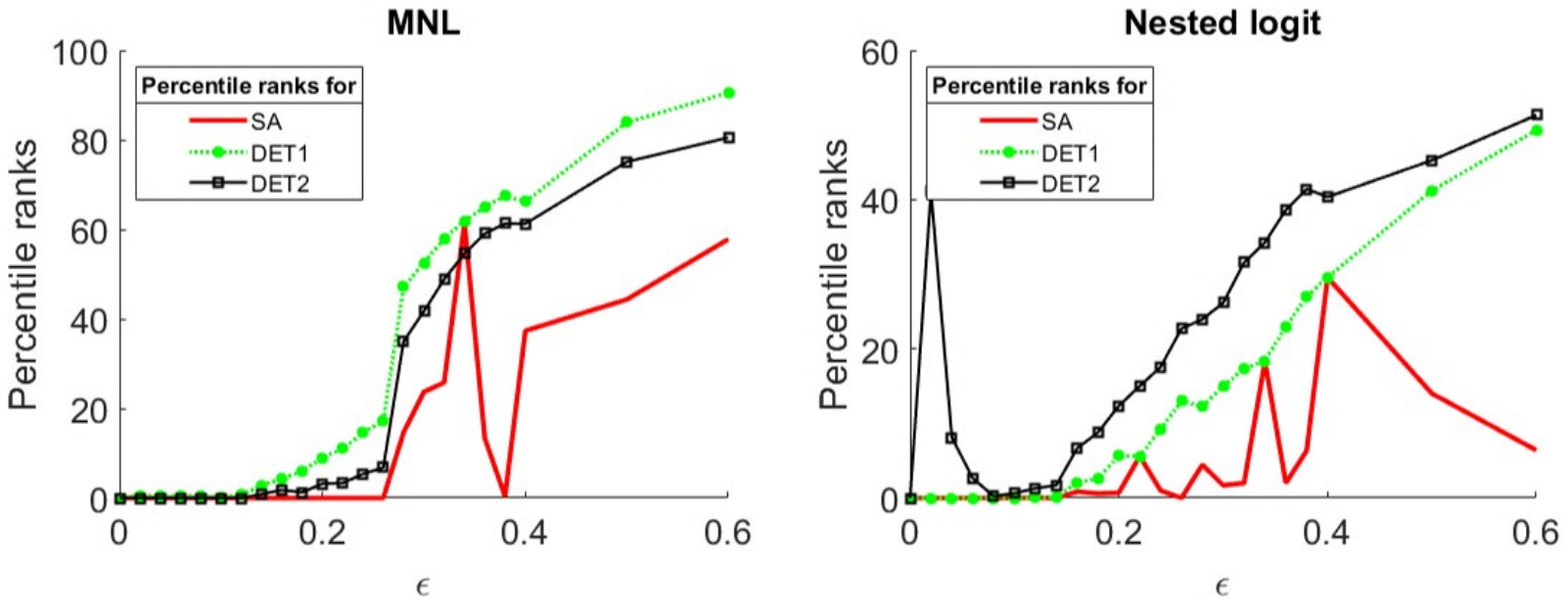}  
  \caption{{The percentile ranks of RO worst value in the distributions given by the SA, DET1, and DET2 solutions under the nested logit choice model for instance with $|I|=1000$ and $m=100$.}}
 \label{fig:1000-100-ranks}
\end{figure}


\paragraph{{Instances of $|I|=82341$ and $m=59$}}
In Figures \ref{fig:82341-59-mnl}, \ref{fig:82341-59-nested}, and  \ref{fig:82341-59-ranks},  we plot similar figures as in the previous sections for large-scale instances of size $|I|=82341$ and $m=59$, which give analogous observations, i.e., the histograms of RO always have lower variances and shorter tails and the percentiles  ranks of the RO's worst-case objective values are always significant, especially when $\epsilon$ becomes large.


\begin{figure}[H]
  \includegraphics[width=\linewidth]{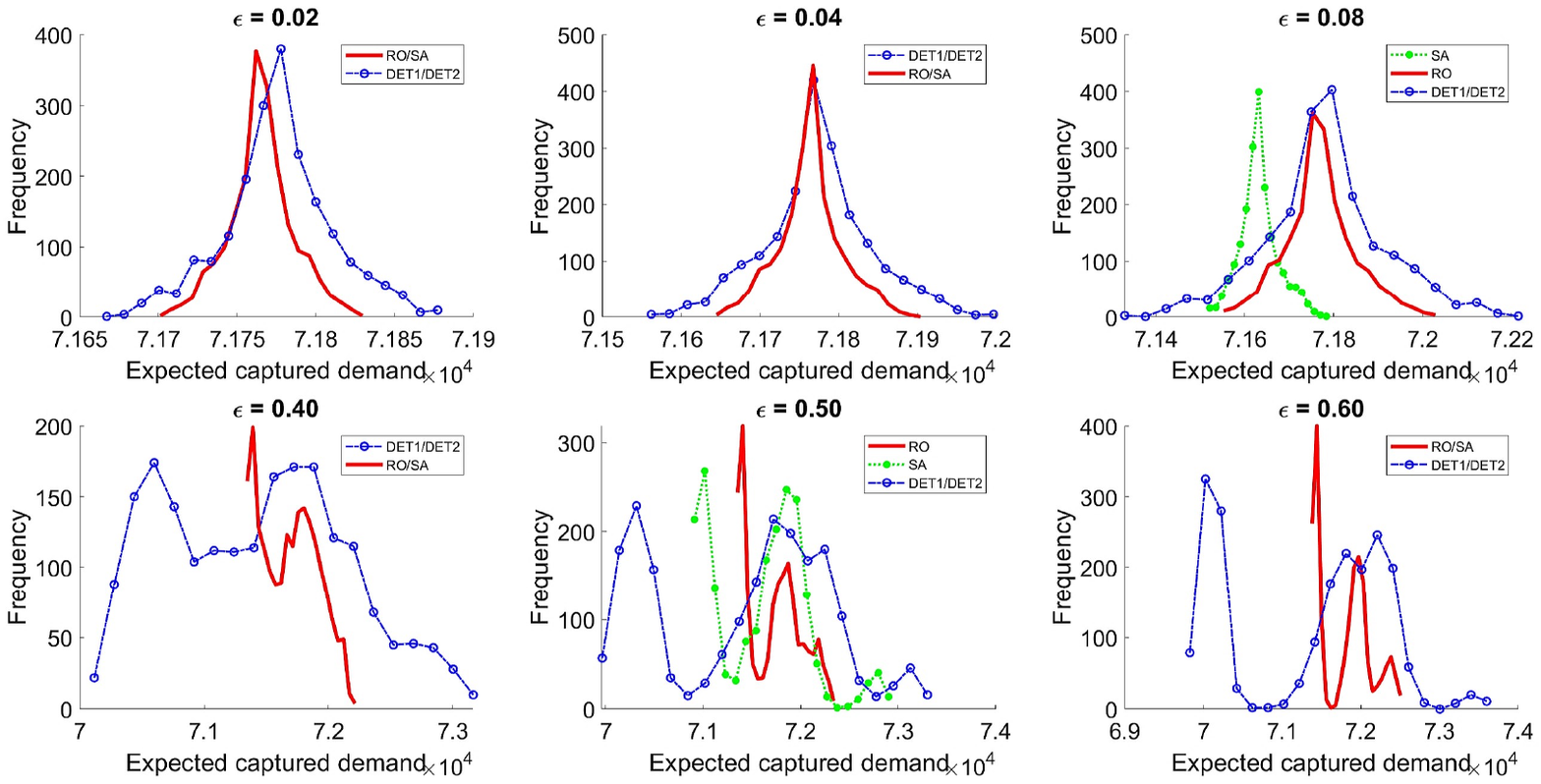}  
  \caption{{Comparison between the distributions of the objective values given by solutions from RO, SA, DET1, and DET2 approaches, under the MNL choice model and with instances of size $|I|=82341$ and $m = 59$.}}
  \label{fig:82341-59-mnl}
\end{figure}

\begin{figure}[H]
  \includegraphics[width=\linewidth]{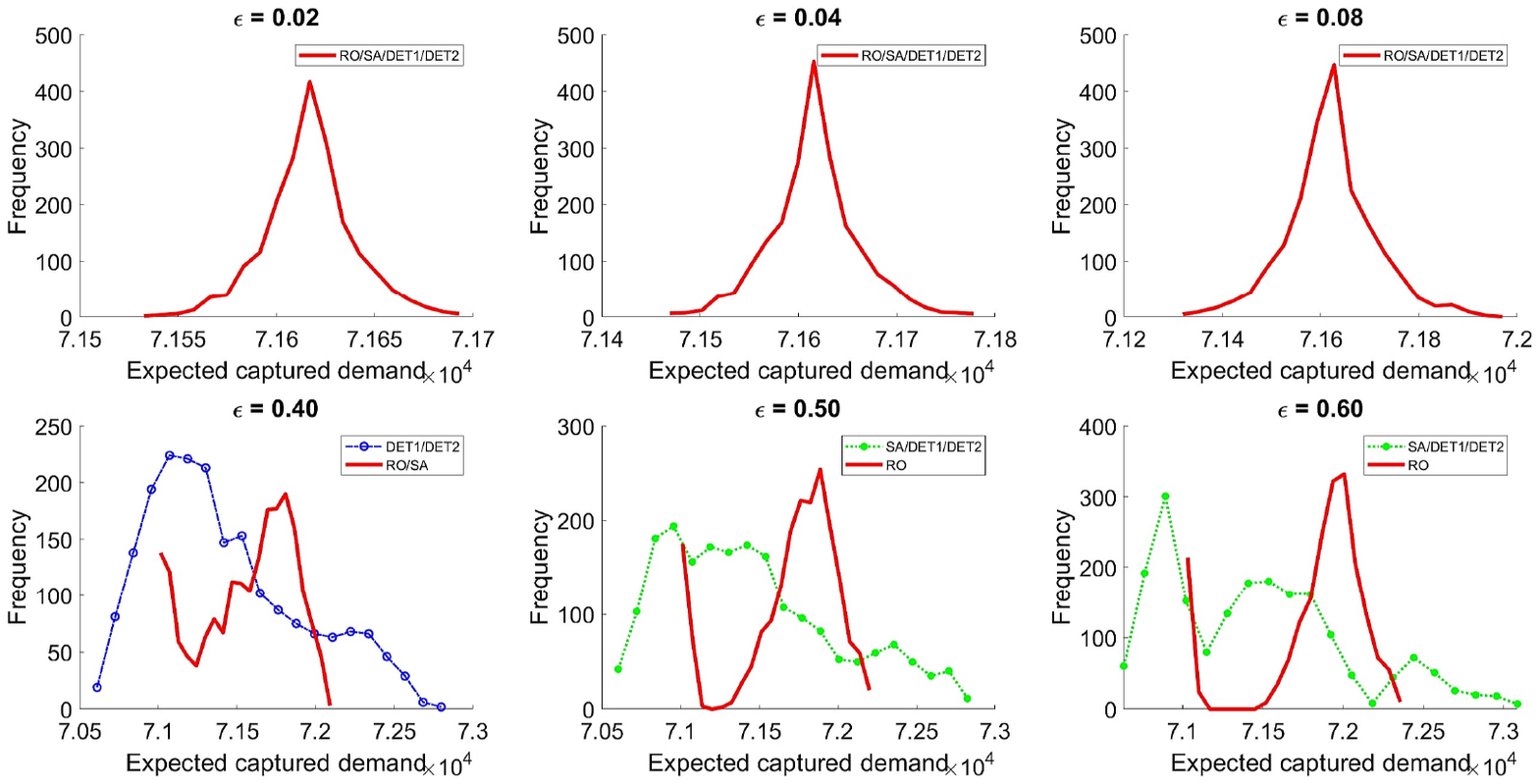}  
  \caption{{Comparison between the distributions of the objective values given by solutions from RO, SA, DET1, and DET2 approaches, under the nested logit choice model and with instances of size $|I|=82341$ and $m = 59$.}}
  \label{fig:82341-59-nested}
\end{figure}

\begin{figure}[H]
  \centering
  \includegraphics[width=0.8\linewidth]{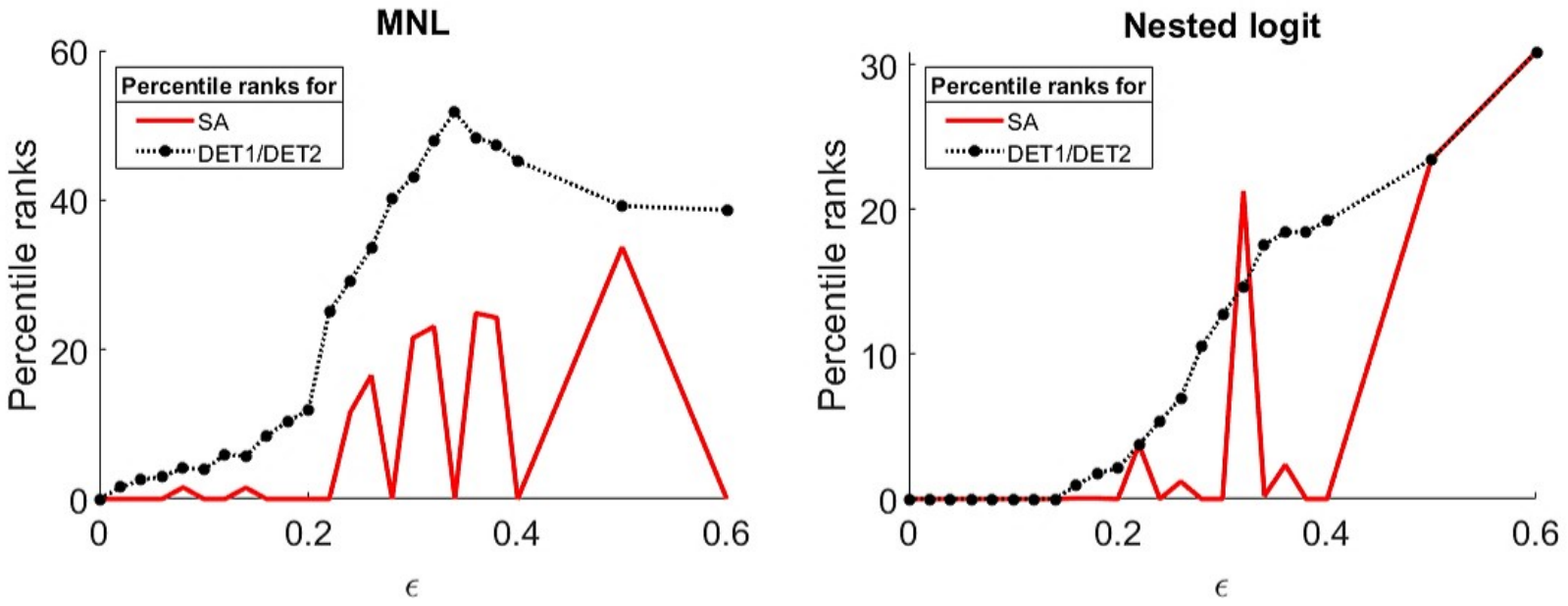}  
  \caption{{The percentile ranks of RO worst value in the distributions given by the SA, DET1, and DET2 solutions under the nested logit choice model for instance with $|I|=82341$ and $m=59$.}}
 \label{fig:82341-59-ranks}
\end{figure}


\end{document}